\documentclass[10pt]{article}
\usepackage{amsmath, amssymb, amsfonts, amsthm, epsfig, float, fdsymbol, graphicx, sectsty}
\usepackage[all]{xy}
\title{Groups with arbitrarily poor permutation stability}
\author{Henry Bradford}

\newtheorem{thm}{Theorem}[section]
\newtheorem{lem}[thm]{Lemma}
\newtheorem{propn}[thm]{Proposition}
\newtheorem{coroll}[thm]{Corollary}
\newtheorem{defn}[thm]{Definition}
\newtheorem{ex}[thm]{Example}
\newtheorem{notn}[thm]{Notation}

\newtheorem{rmrk}[thm]{Remark}
\newtheorem{qu}[thm]{Question}

\DeclareMathOperator{\Alt}{Alt}

\DeclareMathOperator{\GL}{GL}
\DeclareMathOperator{\id}{id}

\DeclareMathOperator{\IRS}{IRS}
\DeclareMathOperator{\Prob}{Prob}

\DeclareMathOperator{\SR}{SR}

\DeclareMathOperator{\Sub}{Sub}
\DeclareMathOperator{\supp}{supp}

\DeclareMathOperator{\Sym}{Sym}

\begin{document}

\maketitle

\begin{abstract}
We propose a quantitative notion of permutation stability for finitely generated groups. Our notion is related to, but distinct from, the ``stability rate'' introduced by Becker and Mosheiff (which is valid within the class of finitely presented groups). 
We construct a family of finitely generated stable groups which exhibit, quantitatively, arbitrarily ``bad'' permutation stability. This means that any application of a ``sample-and-substitute'' algorithm will be very slow in ascertaining whether a given tuple of 
permutations satisfy the defining relations of our groups. 
\end{abstract}

\section{Introduction}


A group $\Gamma$ is \emph{permutation stable} (henceforth simply \emph{stable}) 
if every almost-action of $\Gamma$ on a finite set is a small perturbation 
of an honest action. 
In this paper we study a quantitative notion of stability. 
Intuitively, the group $\Gamma$ is ``very stable'' if an almost-action which is more than a tiny perturbation of an action 
must fail the defining properties of a group action at very many points. 
Conversely, the stability is ``poor'' if there are almost-actions on finite sets 
which come very close to satisfying the defining properties of an action 
but must be perturbed a great deal to yield a genuine action. 
The main result of this paper is that the class of finitely generated 
stable groups contains groups whose stability is arbitrarily poor.

\subsection{Statement of results} \label{StateResultsSubsect}

Throughout, $\Gamma$ is a finitely generated group; 
$\mathbb{F}$ is a finite-rank free group on basis $X$, 
and $\pi : \mathbb{F} \twoheadrightarrow \Gamma$ is a fixed epimorphism. 
For a finite set $\Omega$, let $d_{\Omega}$ denote 
the \emph{normalised Hamming metric} on the symmetric group $\Sym(\Omega)$, 
given by: 
$$d_{\Omega}(\sigma,\tau) = \frac{1}{\lvert \Omega \rvert} \lvert \lbrace \omega \in \Omega : \sigma(\omega) \neq \tau(\omega) \rbrace \rvert.$$

\begin{defn} \label{StableDefn}
We say that $\Gamma$ is stable if, for every $\epsilon > 0$, 
there exists $\delta \in (0,1]$ and a finite set $E \subseteq \ker (\pi)$ such that, 
for any finite set $\Omega$ and any homomorphism 
$\rho : \mathbb{F} \rightarrow \Sym(\Omega)$, if: 
\begin{equation} \label{TestIneq}
d_{\Omega} (\rho(r),\id_{\Omega}) < \delta
\end{equation}
for all $r \in E$, 
then there is a homomorphism $\phi : \Gamma \rightarrow \Sym(\Omega)$ 
satisfying $d_{\Omega} (\rho(x),(\phi \circ \pi)(x))<\epsilon$ 
for all $x \in X$ 
(we say in this case that $\rho$ and $\phi \circ \pi$ are \emph{$\epsilon$-close}). 
\end{defn}

To motivate our formulation of quantitative stability, 
suppose that we know $\Gamma$ is stable; that we have an action of 
$\mathbb{F}$ on a finite set $\Omega$ (via $\rho$), 
and that we would like to test whether $\rho$ is close to a homomorphism 
that factors through $\pi$. 
We may use stability to check this, as follows. 
Given $\epsilon > 0$ and corresponding $(\delta,E)$ 
as in Definition \ref{StableDefn} above, 
if $\rho$ is \emph{not} $\epsilon$-close to 
an action factoring through $\pi$, then for some $r \in E$, 
(\ref{TestIneq}) fails. 
For each $r \in E$, pick $\omega_1 , \ldots , \omega_n \in \Omega$ uniformly at random, with replacement, 
and compute for each $i$ whether or not $\rho(r)[\omega_i] = \omega_i$. 
If we have equality for all $i$, then we say that ``$r$ passes the test''. 
If (\ref{TestIneq}) fails then the probability of 
$r$ passing the test is at most $(1-\delta)^n$. 
To ensure that with high (say $99$ percent) probability, 
we do not allow any $r$ for which (\ref{TestIneq}) fails 
to pass the test, we should take $n$ to be at least proportional to $1/\delta$. 
If all $r \in E$ pass the test, then 
we can conclude with high probability that $\rho$ is 
$\epsilon$-close to some appropriate $\phi \circ \pi$. 
For any given $r \in E$ and $\omega \in \Omega$, 
we may imagine that the computational cost of determining $\rho(r)[\omega]$ 
is proportional to the word-length $\lvert r \rvert$ of $r$ (since we may imagine that evaluating 
$\rho(x)[\omega]$ for a single basis element 
$x\in X$ has constant cost). 
Therefore the total cost of running the test is roughly $\lVert E \rVert / \delta$, 
where: 
\begin{equation*}
\lVert E \rVert = \sum_{r \in E} \lvert r \rvert
\end{equation*}

\begin{defn} \label{StabRate}
The \emph{stability growth function} $F_{\Gamma} ^{\pi} : [1,\infty) \rightarrow [0,\infty)$ 
of $\Gamma$ (with respect to $\pi$) is defined by: 
\begin{equation*}
F_{\Gamma} ^{\pi} (x) = \inf_{(\delta,E)} \big( \lVert E \rVert / \delta\big)
\end{equation*}
where the infimum is taken over all pairs $(\delta,E)$ satisfying 
the conditions of Definition \ref{StableDefn} 
with respect to $\epsilon = 1/x \in (0,1]$. 
\end{defn}

Intuitively, if the function 
$F_{\Gamma} ^{\pi} : [1,\infty) \rightarrow (0,\infty)$ grows slowly, 
then there is a choice of the parameters $(\delta,E)$ 
such that it is computationally efficient to 
perform the test outlined above. 
Our main result shows that there exist finitely generated 
stable groups whose stability is \emph{inefficient}, 
in the sense that the test is very costly to perform 
for any choice of parameters. 

\begin{thm} \label{MainThmIntro}
There is an absolute constant $C>0$ such that, 
for any nondecreasing function $F : [1,\infty) \rightarrow (0,\infty)$ 
there exists a finitely generated stable group $\Gamma$ 
equipped with an epimorphism $\pi : \mathbb{F} \twoheadrightarrow \Gamma$ 
from a finite rank free group $\mathbb{F}$ 
such that for all $x \geq C$, $F_{\Gamma} ^{\pi} (x) \geq F(x)$. 
\end{thm}

\subsection{Background}

Determining, for a given $\pi : \mathbb{F} \twoheadrightarrow \Gamma$, 
whether $\rho : \mathbb{F} \rightarrow \Sym(\Omega)$ is close to some homomorphism 
factoring through $\pi$, 
is an example of a ``property-testing'' problem: 
we are interested in some parameter of a huge object, 
and seek a fast (possibly randomised) algorithm which will, 
in a small number of queries (``small'' meaning, for example, 
of a number which is bounded independent of the size of the object), 
give an approximate evaluation of the parameter. 
In our context, the object in question is a tuple 
$(\rho(x))_{x \in X}$ of permutations of a large finite set $\Omega$, 
and the parameter we wish to evaluate is whether or not 
the $\rho(x)$ satisfy the defining relations between a given ordered set of 
generators $(\pi(x))_{x \in X}$ for the group $\Gamma$. 
The algorithm which implements the test described in Subsection \ref{StateResultsSubsect} above 
is named ``Sample-and-Substitute'' in \cite{BecLubMosI}. 
The number of queries required by Sample-and-Substitute 
is then the number $n$ of points $\omega_i$ which are sampled from the 
permutation domain $\Omega$. 
If the group is stable, then there exists a way of running the algorithm 
with a number of queries which is independent of $\lvert \Omega \rvert$, 
such that the outcome gives reasonable confidence as to whether 
the $\rho(x)$ are close in the Hamming metric to 
permutations satisfying the defining relations of $\Gamma$. 
For more on this perspective on group stability coming from 
theoretical computer science, see \cite{BecLubMos}. 

From this point of view it is very natural to ask for an estimate of 
the computational costs associated to running Sample-and-Substitute, 
of which the stability growth function, as defined above, is a sensible measure. 
Theorem \ref{MainThmIntro} tells us that there are groups 
for which implementing Sample-and-Substitute is viable on very large 
permutation domains $\Omega$, in the sense that the number of queries is 
independent of $\lvert \Omega \rvert$, 
but for which the cost of running Sample-and-Substitute grows very quickly 
as a function of the proximity $\epsilon$ to an action of $\Gamma$ which we demand. 
Currently the class of stable groups for which 
we have good bounds on the stability growth function remains 
extremely limited. 
Becker and Mosheiff \cite{BecMos} quantify the stability 
of finitely generated abelian groups 
(the stability of such groups 
being due to Arzhantseva and Paunescu \cite{ArzPau}). 
In our terminology, their main result implies the following. 

\begin{thm} \label{BeckMoshThmIntro}
    Suppose $\Gamma \cong \mathbb{Z}^d$ for $d \geq 2$. 
    Then there exists $D= D(d) > 0$ such that:
    \begin{center}
    $x^{\frac{d}{2}} \ll_d F_{\Gamma} ^{\pi} (x) \ll_d x^D$
    \end{center}
    for all $x \geq 1$. 
\end{thm}

In particular, $\mathbb{Z}^d$ is 
\emph{polynomially stable}. 
The fact that there are finitely generated groups which are stable 
but not polynomially stable 
is already a novel consequence of our Theorem \ref{MainThmIntro}. 
It remains an intriguing open question 
to construct finitely \emph{presented} 
groups which are stable but not polynomially stable.

The groups $\Gamma$ constructed in the proof of Theorem \ref{MainThmIntro} 
were introduced by the author in \cite{BradRF}, and are in turn a slight 
modification of an earlier construction of Bou-Rabee and Seward \cite{BoRaSewa}; 
we call them the \emph{generalised B.H. Neumann groups} 
(being as they are a variant of the continuum of pairwise nonisomorphic finitely generated 
groups constructed in \cite{Neum}). 
The key to the proof of Theorem \ref{MainThmIntro} is that the 
stability growth function of a group $\Gamma$ controls the disparity 
between two other asymptotic invariants of $\Gamma$: 
the \emph{full residual finiteness growth function} $\mathcal{R}_{\Gamma}$ 
and the \emph{LEF growth function} $\mathcal{L}_{\Gamma}$ 
(see Subsection \ref{LEFRFPrelimSubsect} below). 
We always have $\mathcal{R}_{\Gamma} \geq \mathcal{L}_{\Gamma}$, 
and if the stability growth function of $\Gamma$ grows slowly, then it is not hard to see that
$\mathcal{R}_{\Gamma}$ cannot grow much faster than $\mathcal{L}_{\Gamma}$. 
Bou-Rabee and Seward's groups were introduced precisely as the 
first examples of groups which are residually finite, 
but for which $\mathcal{R}_{\Gamma}$ can grow arbitrarily quickly, 
and a bound on the LEF growth of these groups was already exhibited 
by the author in \cite{Brad}. 
Bounds on $\mathcal{R}_{\Gamma}$ for the groups considered in the present article 
are given in \cite{BradRF}; we give below an upper bound on their LEF growth 
(see Proposition \ref{LEFUBProp}). 

The same basic strategy outlined above has already been fruitfully employed 
by Dogon, Levit and Vigdorovich in \cite{DogLevVig}, 
and applied to the same groups from \cite{BradRF}, 
to prove a result analogous to our Theorem \ref{MainThmIntro}, 
but with respect to \emph{Hilbert-Schmidt stability} 
rather than permutation stability, 
that is with finite symmetric groups equipped with the Hamming metric 
being replaced by unitary groups of finite-dimensional complex 
Hilbert spaces, equipped with the normalised Hilbert-Schmidt metric. 
It is interesting to note that HS-stability of the groups 
from \cite{BradRF} follows from a general result, also proved in \cite{DogLevVig}, 
on the HS-stability of so-called \emph{diagonal product} groups. 
Such a result is not known in the setting of permutation stability, 
so we make use of more ad hoc arguments. 
It should also be noted that Dogon, Levit and Vigdorovich state their result 
only for the ``stability radius growth'' function: 
this is a variant of stability growth function, according to which only the maximal word 
length of elements of the finite set $E$ appearing in Definition \ref{StableDefn} 
contributes to the ``cost'', with the parameter $\delta$ playing no r\^{o}le. 
Nevertheless, their proof should also be relevant to the 
analogue, for the Hilbert-Schmidt metric, of 
the stability growth function as we have defined it. 

Having established bounds on $\mathcal{R}_{\Gamma}$ and $\mathcal{L}_{\Gamma}$ for our groups 
$\Gamma$, it remains to prove that the $\Gamma$ are indeed stable. 
The original B.H. Neumann groups were proven stable by Levit and Lubotzky \cite{LevLub}; 
our strategy follows theirs: first, since the $\Gamma$ are \emph{amenable groups}, 
the criterion of Becker, Lubotzky and Thom \cite{BecLubTho} 
for stability of $\Gamma$ in terms of the cosoficity of the 
\emph{invariant random subgroups} (IRS) of $\Gamma$ is available to us 
(see Theorem \ref{IRSStabCritThm} below). 
To prove cosoficity of many of the IRS of $\Gamma$, 
we avail ourselves of the technique of \emph{Weiss approximation} 
used in \cite{LevLub} (see Theorem \ref{WeissThm} below). 
Finally, the structure of $\Gamma$ as a diagonal product of marked groups 
which converge to a ``limiting group'' $\tilde{\Gamma}$ 
(which is hence a quotient of $\Gamma$) allows us to lift 
restrictions on the structure of the IRS of $\tilde{\Gamma}$ 
to those of $\Gamma$. In our setting, $\tilde{\Gamma}$ 
is always the \emph{ternary lamplighter group} $C_3 \wr \mathbb{Z}$, 
whose IRS were studied by Bowen, Grigorchuk and Kravchenko \cite{BowGrigKrav}. 

There are many other reasonable formulations of stability growth function, 
which weigh the costs associated to the parameters 
$E$ and $\delta$ appearing in Definition \ref{StableDefn} differently. 
Our methods are extremely robust with respect to 
the precise formulation of stability growth function used, 
hence we would have an analogue of Theorem \ref{MainThmIntro} for any such function, 
provided the function adequately encodes the idea that testing longer words is more costly. 

The paper is structured as follow. 
In Section \ref{PrelimSect} we prove the basic properties of the stability growth function 
and compare it with the stability rate function of Becker and Moshieff, 
proving Theorem \ref{BeckMoshThmIntro} in the process. 
We also recall the definitions of the LEF growth and full residual finiteness growth functions 
and prove our key inequality relating the two to the stability growth function. 
In Section \ref{NeumannSection} we describe the generalised B.H. Neumann groups, 
and show that \emph{if} they are stable, then they satisfy 
the conclusion of our main result. 
In Section \ref{IRSSect} we study the invariant random subgroups of the 
generalised B.H. Neumann groups and thereby prove that they \emph{are} stable. 
We close with some questions for further investigation. 

\section{Preliminaries} \label{PrelimSect}

\subsection{The stability growth function}

Once again, we fix a finite-rank free group $\mathbb{F}$ 
on a preferred basis $X$. 

\begin{defn} \label{SolnDefn}
For $E \subseteq \mathbb{F}$ we refer to a homomorphism 
$\phi : \mathbb{F} \rightarrow \Sym(n)$ as a \emph{solution} 
for $E$ if $E \subseteq \ker (\phi)$ and for $\delta > 0$ 
we refer to $\phi$ as a \emph{$(\delta,E)$-almost-solution} 
if, for all $r \in E$, $d_n (\phi(r),\id_n) < \delta$. 
\end{defn}

Thus, given a homomorphism $\pi : \mathbb{F} \rightarrow \Gamma$ 
we may restate Definition \ref{StableDefn} as: 
$\Gamma$ is stable if for every $\epsilon > 0$ 
there exist $\delta > 0$ and $E \subseteq \ker(\pi)$ 
such that every $(\delta,E)$-almost-solution is $\epsilon$-close 
to a solution for $\ker(\pi)$. 

\begin{ex} \label{FreeStabNoCost}
Suppose $\pi : \mathbb{F} \twoheadrightarrow \Gamma$ 
is an isomorphism, so that $\ker (\pi) = \lbrace e \rbrace$. 
Thus every homomorphism $\phi : \mathbb{F} \rightarrow \Sym(n)$ 
is a solution for $\ker (\pi)$. 
In particular for any $\epsilon=1/x , \delta > 0$, 
every $(\delta, \emptyset)$-almost-solution 
is $\epsilon$-close to a solution for $\ker (\pi)$. 
Thus $F_{\Gamma} ^{\pi} (x) = 0$. 
\end{ex}

Apart from the degenerate case described 
in Example \ref{FreeStabNoCost}, 
every epimorphism yields a stability growth function 
which is at least linear 
(compare with Proposition A.5 of \cite{BecMos}, 
which establishes an analogous statement for 
the \emph{stability rate} 
(Definition \ref{StabRateDefn} below) of a finitely presented group). 

\begin{propn} \label{StabLinearProp}
Suppose $\Gamma$ is stable. If $\ker (\pi)$ is nontrivial, 
then there exists $C_1 ,C_2> 0$ such that for all $x \geq C_1$, 
\begin{itemize}
    \item[(i)] If $\delta > 0$ is such that there exists $E \subseteq \ker(\pi)$, 
    for which every $(\delta,E)$-almost-solution is $x^{-1}$-close to 
    a solution for $\ker(\pi)$, then $\delta \leq C_2 x^{-1}$; 

    \item[(ii)] $F_{\Gamma} ^{\pi} (x) \geq x/4$. 
\end{itemize}
\end{propn}

\begin{proof}
First, observe that for any $e \neq w \in \ker (\pi)$ 
and any homomorphism $\phi : \mathbb{F} \rightarrow \Sym(n)$ 
with $\phi(w) \neq \id_n$, 
if $0 < \epsilon <d_n (\phi(w),\id_n)/\lvert w \rvert$ 
then $\phi$ is not $\epsilon$-close 
to a solution for $\ker(\pi)$. 

Thus, choosing an arbitrary $e \neq w_0 \in \ker (\pi)$ 
and a homomorphism $\phi_0 : \mathbb{F} \rightarrow \Sym(n)$ 
with $\phi_0 (w_0) \neq \id_n$, and let 
$C_1 = 2 \lvert w_0 \rvert / d_n (\phi_0 (w_0),\id_n)$. 
Then for $E = \emptyset$ or $E = \lbrace e \rbrace$, 
$\phi_0$ is a $(\delta,E)$-almost solution for every $\delta > 0$ 
but is not $C_1 ^{-1}$-close to a solution for $\ker (\pi)$. 

Now, for every $x \geq C_1$, every $\delta > 0$ 
and every $E \subseteq \ker(\pi)$ finite 
with the property that every $(\delta,E)$-almost-solution 
is $x^{-1}$-close to a solution for $\ker(\pi)$, 
we must have $E \nsubseteq \lbrace e \rbrace$. 
Given such a pair $(\delta,E)$, 
let $e \neq w_1 \in \ker(\pi)$ be an element of minimal word-length, 
and set $E' = E \cup \lbrace w_1 \rbrace$, 
so that $\lVert E' \rVert \leq 2 \lVert E \rVert$ 
and every $(\delta,E')$-almost-solution 
is still a $(\delta,E)$-almost-solution. 
Set $C_2 = 2 \lvert w_1 \rvert$ 
and take a homomorphism $\psi : \mathbb{F} \rightarrow \Sym(\Omega)$, 
for some finite set $\Omega$, for which $\psi (w_1) \neq \id_{\Omega}$. 
Composing with the regular permutation representation of 
$\Sym(\Omega)$ if required, we may assume 
$d_{\Omega} (\psi (w_1),\id_{\Omega})=1$. 
Next, choose positive integers $q$ and $r$ such that: 
$$\frac{q \lvert \Omega \rvert}{q \lvert \Omega \rvert + r} 
\in (\delta/2,\delta),$$
let $\Omega_1 , \ldots , \Omega_q$ be disjoint 
copies of $\Omega$ and let $\Pi$ be a set of size $r$, 
disjoint from the $\Omega_i$. 
Let $\Sigma = \Omega_1 \cup \ldots \cup \Omega_q \cup \Pi$ 
and let $\Psi : \mathbb{F} \rightarrow \Sym(\Sigma)$ 
be defined such that for each $v \in \mathbb{F}$, 
$\Psi(v)$ fixes all points of $\Pi$ 
and acts on each $\Omega_i$ as $\psi (v)$ acts on $\Omega$. 
Then $\Psi$ is a $(\delta,E')$-almost-solution 
but $d_{\Sigma} (\Psi(w_1),\id_{\Sigma}) > \delta/2$. 
Thus, by our initial observation, 
$\Psi$ is not $(\delta/2\lvert w_1 \rvert)$-close 
to a solution for $\ker (\pi)$. Hence: 
$$x \leq \frac{2\lvert w_1 \rvert}{\delta} \leq 2 (\frac{\lVert E' \rVert}{\delta}) 
\leq 4 (\frac{\lVert E \rVert}{\delta}) $$
Since this holds for every such pair $(\delta,E)$, 
we have $x \leq 4 F_{\Gamma} ^{\pi} (x)$, and (ii) follows. 
Moreover, the leftmost inequality yields (i). 
\end{proof}

A priori, 
for a given group $\Gamma$, 
satisfaction of Definition \ref{StableDefn} 
and the behaviour of the stability growth function $F_{\Gamma} ^{\pi}$ 
depends on a choice of finite-rank free group 
$\mathbb{F}$ and of an epimorphism 
$\pi : \mathbb{F} \twoheadrightarrow \Gamma$. 
In fact, stability of the group $\Gamma$ 
is known to be an intrinsic property of $\Gamma$ itself 
(see Remark 3.13 of \cite{BecLubTho}). 
Moreover, Becker and Mosheiff show that the stability 
rate of a finitely presented group is independent 
of choice of presentation up to equivalence 
of growth functions (see Section 2 of \cite{BecMos}), 
in the following sense which is standard in the geometric group theory literature. 

\begin{defn}
    Let $F_1 , F_2 : [0,\infty) \rightarrow (0,\infty)$ be nondecreasing functions. 
    We write $F_1 \preceq F_2$ ($F_2$ \emph{dominates} $F_1$) 
    if there exists $C>0$ such that 
    for all $x \geq 1$, 
    $$F_1 (x) \leq C F_2 (Cx) + Cx + C$$
    and write $F_1 \approx F_2$ 
    ($F_1$ and $F_2$ are equivalent) if both $F_1 \preceq F_2$ and $F_2 \preceq F_1$. 
\end{defn}

By adapting the arguments of of \cite{BecMos} we show that 
up to equivalence, 
the stability growth function $F_{\Gamma} ^{\pi}$ 
is independent of $\pi$. 

\begin{propn} \label{PresnEquivPropn}
Let $\Gamma$ be a finitely generated group, 
and for $i=1,2$ let $\mathbb{F}_i$ be a 
free group on a finite basis $X_i$ and 
let $\pi_i : \mathbb{F}_i \twoheadrightarrow \Gamma$ 
be an epimorphism. 
If $\pi_1$ satisfies the conditions of Definition \ref{StableDefn}, 
then so too does $\pi_2$. Moreover there exists $C_0 > 0$ such that for all $x \geq 1$, 
\begin{itemize}
    \item[(i)] If $\pi_1$ is an isomorphism then $F_{\Gamma} ^{\pi_2} (x) \leq C_0 x$; 
    \item[(ii)] If $\pi_1$ is not an isomorphism then $F_{\Gamma} ^{\pi_2} (x) \leq C_0 F_{\Gamma} ^{\pi_1} (C_0 x)$. 
\end{itemize}
\end{propn}

\begin{proof}
Enlarging the constant $C_0$ if required, it suffices to prove the given inequalities for 
$x$ sufficiently large. 
By the universal property of free groups, there exist 
homomorphisms $\varphi : \mathbb{F}_1 \rightarrow \mathbb{F}_2$ 
and $\psi : \mathbb{F}_2 \rightarrow \mathbb{F}_1$ 
satisfying $\pi_2 \circ \varphi = \pi_1$ 
and $\pi_1 \circ \psi = \pi_2$. 
Let $C>1$ be such that $\varphi(X_1) \subseteq B_{X_2} (C)$ 
and $\psi(X_2) \subseteq B_{X_1} (C)$. 
Let $$Y_2 = \lbrace x^{-1} (\varphi \circ \psi)(x) : x \in X_2 \rbrace 
\subseteq \ker (\pi_2) \cap B_{X_2} (C^2 + 1).$$
Let $\epsilon > 0$ and take $\delta > 0$ and $E_1\subseteq\ker(\pi_1)$ 
finite such that every $(\delta,E_1)$-almost-solution is 
$(\epsilon/2C)$-close to a solution for $\ker(\pi_1)$. 
Moreover, if $\pi_1$ is an isomorphism we take $E_1 = \emptyset$ 
and $\delta = 1$. 

Set $E_2 ' = \varphi(E_1) \cup Y_2$ and $\delta' = \min (\delta,\epsilon/2)$, 
and let $\tilde{\rho} : \mathbb{F}_2 \rightarrow \Sym(n)$ 
be an $(\delta',E_2 ' )$-almost-solution. 
Define $\overline{\rho} = \tilde{\rho} \circ \varphi : \mathbb{F}_1 
\rightarrow \Sym(n)$. 
Then since $\varphi(E_1) \subseteq E_2 '$, 
$\overline{\rho}$ is a $(\delta',E_1)$-almost-solution, 
and hence is $(\epsilon/2C)$-close to a solution $\rho : \mathbb{F}_1 \rightarrow \Sym(n)$ for $\ker(\pi_1)$  
(as $\delta' \leq \delta$). 
Thus $\rho \circ \psi : \mathbb{F}_2 \rightarrow \Sym(n)$ 
is a solution for $\ker(\pi_2)$ 
(note that in the case for which $\pi_1$ is an isomorphism, 
$\overline{\rho}$ is already a solution for $\ker(\pi_1) = \emptyset$, 
so we may take $\rho = \overline{\rho}$). 

For $x \in X_2$, $x$ and $(\varphi \circ \psi) (x)$ differ by 
an element of $Y_2 \subseteq E_2 '$, 
so since $\delta' < \epsilon/2$, 
$\tilde{\rho}$ is $\epsilon/2$-close to 
$\tilde{\rho} \circ (\varphi \circ \psi) = \overline{\rho} \circ \psi$, 
which is in turn $\epsilon/2$-close to $\rho \circ \psi$, 
since $\psi(X_2) \subseteq B_{X_1} (C)$. 
Since this holds for every $\epsilon > 0$, we conclude that 
$\pi_2$ satisfies the conditions of Definition \ref{StableDefn}, and moreover: 
\begin{equation} \label{ComparePresnIneq}
    F_{\Gamma} ^{\pi_2} (1/\epsilon) \leq \frac{\lVert E_2 ' \rVert}{\delta'}  
\leq (\frac{1}{\delta'}) \big( (C^2 + 1) \lvert X_2 \rvert + C \lVert E_1 \rVert \big).
\end{equation}
We distinguish two cases. 
If $\pi_1$ is an isomorphism, then by our choice above, $E_1 = \emptyset$ 
and $\delta' = \epsilon/2$, so (\ref{ComparePresnIneq}) becomes: 
$$F_{\Gamma} ^{\pi_2} (1/\epsilon) \leq C'/\epsilon$$
for $C' = (C^2 + 1) \lvert X_2 \rvert$. If $\pi_1$ is not an isomorphism, 
then we may apply Proposition \ref{StabLinearProp} (i) to obtain 
$2 C \delta / C_2 \leq  \epsilon $, so that $\delta' = c' \delta$ 
(for $c' = \min(1,C/C_2)$), and since $\lVert E_1 \rVert \geq 1$ (\ref{ComparePresnIneq}) yields: 
$$F_{\Gamma} ^{\pi_2} (1/\epsilon) \leq C'' (\frac{\lVert E_1 \rVert}{\delta})$$
(for $C'' = ((C^2 + 1) \lvert X_2 \rvert + C)/c'$), 
and taking the infimum over all such pairs $(E_1,\delta)$ the result follows. 
\end{proof}

\begin{coroll} \label{PresnEquivCoroll}
    Let $\pi_i : \mathbb{F}_i \twoheadrightarrow \Gamma$ ($i=1,2$) 
    be as in Proposition \ref{PresnEquivPropn}. 
    Then $F_{\Gamma} ^{\pi_1} \approx F_{\Gamma} ^{\pi_2}$. 
\end{coroll}

\begin{proof}
    If one of $\pi_1$ or $\pi_2$ satisfy Definition \ref{StableDefn} 
    then by Proposition \ref{PresnEquivPropn}, both do 
    (and if this is not the case there is nothing to prove). 
    If one of $\pi_1$ or $\pi_2$ is an isomorphism, 
    then the conclusion follows from Example \ref{FreeStabNoCost} 
    and Proposition \ref{PresnEquivPropn} (i). 
    If not, then the conclusion follows from 
    Proposition \ref{PresnEquivPropn} (ii) (applied twice). 
\end{proof}

With Corollary \ref{PresnEquivCoroll} in hand, 
we may speak unambiguously about groups of \emph{polynomial stability growth}; 
\emph{exponential stability growth}, and so on. 

If $\Gamma$ is a stable group and has a finite group presentation $\langle X \mid R \rangle$ 
(that is, if $\ker(\pi)$ is normally generated by the finite set $R \subseteq \ker(\pi)$), 
then $\Gamma$ satisfies the following property, which is a priori stronger than 
Definition \ref{StableDefn}: for all $\epsilon > 0$ there exists $\delta > 0$ 
such that any $(\delta,R)$-almost solution is $\epsilon$-close to a solution for $\ker(\pi)$ 
(see Section 3 of \cite{BecLubTho}). 
Making this latter condition effective 
(just as the stability growth function makes Definition \ref{StableDefn} effective) 
is achieved in \cite{BecMos} via the introduction of the following function. 

\begin{defn} \label{StabRateDefn}
    Let $\langle X \mid R \rangle$ be a finite group presentation. 
    The \emph{stability-rate} of $\langle X \mid R \rangle$ is 
    the function $\SR_{\langle X \mid R \rangle} : (0,\lvert R \rvert ] \rightarrow [0,\lvert X \rvert]$ given by: 
    $$\SR_{\langle X \mid R \rangle} (\epsilon) = \sup \big\lbrace G_R (\rho) : \rho : \mathbb{F} \rightarrow \Sym (\Omega), \Omega \text{ a finite set, } L_R(\rho) \leq \epsilon \big\rbrace$$
    where $G_R (\rho)$ denotes the \emph{global defect}: 
    $$G_R (\rho) = \inf \big\lbrace \sum_{x \in X} d_{\Omega} (\rho(x),\phi(x)) 
    : \phi \text{ is a solution for $R$} \big\rbrace$$
    of $\rho$, and $L_R (\rho)$ is the \emph{local defect}: 
    $$L_R (\rho) = \sum_{r \in R} d_{\Omega} (\rho(r),\id_{\Omega}).$$
\end{defn}

To achieve better consistency with our terminology, we introduce the following 
variant of stability rate. 

\begin{notn}
Let $\langle X \mid R \rangle$ be a finite group presentation for the stable group $\Gamma$. 
Define $H_{\langle X \mid R \rangle} : [1,\infty) \rightarrow [0,\infty)$ such 
that $H(x)$ is the infimal value of $1/\delta$, as $\delta > 0$ ranges over all such values 
for which every $(\delta,R)$-almost solution is $x^{-1}$-close to a solution for $\ker(\pi)$. 
\end{notn}

It is immediate from definitions that if $H_{\langle X \mid R \rangle}$ 
grows slowly, then so too does $F_{\Gamma}$ (we may simply take $E=R$ 
in Definition \ref{StableDefn}). 
A converse inequality is less straightforward, 
since given a finite subset $E \subseteq \ker(\pi)$ one may require a very long 
word in the conjugates of elements of $R^{\pm 1}$ to express the elements of $E$. 
In this connection, recall that the \emph{Dehn function} 
$D_{\langle X \mid R \rangle} : \mathbb{N} \rightarrow \mathbb{N}$ 
of the finite presentation $\langle X \mid R \rangle$ is defined 
such that $D_{\langle X \mid R \rangle} (n)$ is the maximal area 
of an element $w \in \ker(\pi)$ satisfying $\lvert w \rvert \leq n$, 
where the \emph{area} of $w$ is the minimal $m \in \mathbb{N}$ 
such that $w = (u_1 r_1 u_1 ^{-1}) \cdots (u_m r_m u_m ^{-1})$ 
for some $u_i \in \mathbb{F}$ and $r_i \in R^{\pm 1}$. 
Note that we may extend the domain of $D_{\langle X \mid R \rangle}$ 
to $[1,\infty)$ by setting $D_{\langle X \mid R \rangle}(x) = D_{\langle X \mid R \rangle}(\lfloor x\rfloor)$. As such, $D_{\langle X \mid R \rangle} : [1,\infty) \rightarrow \mathbb{N}$ 
is nondecreasing and right-continuous. 

\begin{propn} \label{HvsStabGrowthProp}
Let $\Gamma$ be a stable group with finite presentation $\langle X \mid R \rangle$ 
(corresponding to the epimorphism $\pi : \mathbb{F} \rightarrow \Gamma$). 
If $\pi$ is not an isomorphism, then for all $x \geq 1$, 
\begin{equation} \label{HvsStabGrowthPropEqn}
\frac{1}{\lVert R \rVert}F_{\Gamma} ^{\pi} (x) \leq  H_{\langle X \mid R \rangle} (x) 
\leq F_{\Gamma} ^{\pi} (x) D_{\langle X \mid R \rangle}(F_{\Gamma} ^{\pi} (x)). 
\end{equation}
If $\pi$ is an isomorphism (so that $R = \emptyset$ or $\lbrace e_{\mathbb{F}} \rbrace$) 
then $H_{\langle X \mid R \rangle} (x) = 0$ for all $x \geq 1$. 
\end{propn}

\begin{proof}
If $\pi$ is an isomorphism, then for every $x \geq 1$, 
every homomorphism $\phi : \mathbb{F} \rightarrow \Sym(n)$ 
is a solution for $R$, hence is $(\delta,R)$-almost solution for all $\delta > 0$, 
and is $x^{-1}$-close to a solution for $\ker(\pi)$. 
Thus we may assume that $\pi$ is not injective. 

First, given $x \geq 1$, and $\delta > 0$ such that every $(\delta,R)$-almost solution 
is $x^{-1}$-close to a solution for $\ker(\pi)$, 
we may take $E=R$ in Definition \ref{StableDefn} and conclude: 
$$F_{\Gamma} ^{\pi} (x) \leq \lVert R \rVert / \delta$$
so that, taking the infimum of $\lVert R \rVert / \delta$ over all such $\delta$, 
we have the first half of (\ref{HvsStabGrowthPropEqn}). 

Conversely, given $x \geq 1$, suppose we have $\delta > 0$ and a finite subset $E \subseteq \ker(\pi)$ which satisfy Definition \ref{StableDefn} with respect to 
$\epsilon = x^{-1}$. Let $\delta' > 0$ and let $\rho : \mathbb{F} \rightarrow \Sym(\Omega)$ 
be a $(\delta',R)$-almost solution. 
Given $w \in E$, write $w = (u_1 r_1 u_1 ^{-1}) \cdots (u_m r_m u_m ^{-1})$ 
with $u_i \in \mathbb{F}$ and $r_i \in R^{\pm 1}$, 
and with $m \leq D_{\langle X \mid R \rangle} (\lvert w \rvert) \leq D_{\langle X \mid R \rangle} (\lVert E \rVert) $. Then: 
\begin{align*}
d_{\Omega} \big( \rho(w),\id_{\Omega} \big) 
& \leq \sum_{i=1} ^m d_{\Omega} \big( \rho(u_1 r_1 u_1 ^{-1}) \cdots \rho(u_i r_i u_i ^{-1}),\rho(u_1 r_1 u_1 ^{-1}) \cdots \rho(u_{i-1} r_{i-1} u_{i-1} ^{-1}) \big) \\
& =  \sum_{i=1} ^m d_{\Omega} \big( \rho(u_i) \rho(r_i) \rho(u_i) ^{-1},\id_{\Omega} \big) \\
& = \sum_{i=1} ^m d_{\Omega} \big( \rho(r_i) ,\id_{\Omega} \big) \\
& < \delta' D_{\langle X \mid R \rangle} (\lVert E \rVert)
\end{align*}
(by bi-invariance of the metric $d_{\Omega}$). 
Thus, provided $\delta' \leq \delta / D_{\langle X \mid R \rangle} (\lVert E \rVert)$, 
$\rho$ is also a $(\delta,E)$-almost solution, 
and hence is $\epsilon$-close to a solution for $\ker(\pi)$. We conclude: 
\begin{equation} \label{DehnIneq}
H_{\langle X \mid R \rangle} (x) \leq D_{\langle X \mid R \rangle} (\lVert E \rVert) / \delta
\end{equation}
and by monotonicity of $D_{\langle X \mid R \rangle}$, we may bound 
that right-hand side of (\ref{DehnIneq}) above by 
$D_{\langle X \mid R \rangle} (\lVert E \rVert/ \delta) (\lVert E \rVert / \delta)$, 
and since the function $(x \mapsto x D_{\langle X \mid R \rangle} (x))$ is nondecreasing 
and right-continuous, 
we may take the infimum over all such pairs $(\delta,E)$ 
to obtain the second half of (\ref{HvsStabGrowthPropEqn}). 
\end{proof}

\begin{rmrk} \label{BetterDehnBdRmrk}
    The bound (\ref{DehnIneq}) will commonly allow us to obtain 
    a stronger conclusion than (\ref{HvsStabGrowthPropEqn}). 
    Provided $D_{\langle X \mid R \rangle}$ satisfies: 
    $$D_{\langle X \mid R \rangle}(x)y \leq D_{\langle X \mid R \rangle}(xy)$$
    for all $x \geq C$ and $y \geq 1$ for some constant $C \geq 1$, 
    we may bound the right-hand side 
    of (\ref{DehnIneq}) above by $D_{\langle X \mid R \rangle} (\lVert E \rVert/ \delta)$, 
    for all $x$ larger than a constant. 
    Concluding as in the proof of Proposition \ref{HvsStabGrowthProp}, 
    we have: 
    \begin{equation}
      H_{\langle X \mid R \rangle} (x) \leq C D_{\langle X \mid R \rangle}(F_{\Gamma} ^{\pi} (x))
    \end{equation}
    for some constant $C>0$. 
\end{rmrk}

We may relate $H_{\langle X \mid R \rangle}$ to $\SR_{\langle X \mid R \rangle}$ 
as follows. 

\begin{lem} \label{SRvsHLem}
Let $\Gamma$ and  $\langle X \mid R \rangle$ be as in Proposition \ref{HvsStabGrowthProp}. 
Suppose $R \nsubseteq \lbrace e_{\mathbb{F}} \rbrace$. 
\begin{itemize}
    \item[(i)] There exists $C_1 > 0$ such that for all $x \geq C_1$ we have: 
    $$\SR_{\langle X \mid R \rangle} \big(\frac{1}{2 H_{\langle X \mid R \rangle}(x)}\big) \leq \frac{\lvert X \rvert}{x};$$ 

    \item[(ii)] There exists $c_2 > 0$ such that for all $\delta \in (0,c_2)$ we have: 
    $$H_{\langle X \mid R \rangle} \big( \frac{1}{\SR_{\langle X \mid R \rangle} (\lvert R \rvert \delta)} \big)
     \leq \frac{1}{\delta}.$$
    
\end{itemize}
\end{lem}

\begin{proof}
First, note that, for any $y > 0$ 
and any homomorphism $\rho : \mathbb{F} \rightarrow \Sym(\Omega)$, if $G_R (\rho) \leq y$, 
then $\rho$ is $y$-close to a solution for $R$, and hence for $\ker(\pi)$ 
(since $R$ normally generates $\ker(\pi)$), and conversely, 
if there is a solution for $\ker(\pi)$ which is $y$-close to $\rho$, 
then $G_R (\rho) \leq \lvert X \rvert y$. 
Similarly, if $L_R (\rho) \leq y$ then $\rho$ is a $(y,R)$-almost solution, 
and conversely, if $\rho$ is a $(y,R)$-almost solution, 
then $L_R (\rho) < \lvert R \rvert y$. 

Now, as $R \nsubseteq \lbrace e_{\mathbb{F}} \rbrace$ we have that $\pi$ is not an isomorphism. 
Write $H_{\langle X \mid R \rangle}(x) = h_x$. 
Letting $C_1>0$ be as in Proposition \ref{StabLinearProp}, 
we have $h_x \neq 0$ for all $x \geq C_1$ 
(by Proposition \ref{StabLinearProp} (ii) and Proposition \ref{HvsStabGrowthProp}). 
Thus, since $2 h_x > h_x$ we have that every $((\frac{1}{2 h_x}),R)$-almost-solution 
is $x^{-1}$-close to a solution for $\ker(\pi)$. 
Thus, for any $\rho : \mathbb{F} \rightarrow \Sym(\Omega)$ satisfying $L_R(\rho) \leq \frac{1}{2 h_x}$ 
we have $G_R(\rho) \leq \frac{\lvert X \rvert}{x}$, so that, taking the supremum over all such $\rho$, 
we obtain (i). 

For (ii), note that by Proposition A.5 of \cite{BecMos} we have $\SR_{\langle X \mid R \rangle}(\gamma) > 0$ 
for all $\gamma > 0$. Let $\delta > 0$ and set $x = \frac{1}{\SR_{\langle X \mid R \rangle}(\lvert R \rvert \delta)} \geq 1$. 
By stability of $\Gamma$, there exists $c_2 > 0$ such that for all $\delta \in (0,c_2)$ we have $x \geq 1$ 
(see Definition 1.5 of \cite{BecMos}). 
For such $\delta$, if $\rho : \mathbb{F} \rightarrow \Sym(\Omega)$ is a $(\delta,R)$-almost-solution, 
then $L_R(\rho) < \lvert R \rvert \delta$ so $G_R(\rho) \leq \SR_{\langle X \mid R \rangle}(\lvert R \rvert \delta)$, thus $\rho$ is $\SR_{\langle X \mid R \rangle}(\lvert R \rvert \delta)$-close to a solution, 
as required. 
\end{proof}

\begin{ex} \label{PolySREquiv}
    Let $\alpha \geq 1$. By Lemma \ref{SRvsHLem} we have: 
    \begin{itemize}
        \item[(i)] $H_{\langle X \mid R \rangle}(x) \ll x^{\alpha}$ 
        iff $\SR_{\langle X \mid R \rangle}(\epsilon) \ll \epsilon^{\frac{1}{\alpha}}$; 
        
        \item[(ii)] $H_{\langle X \mid R \rangle}(x) \gg x^{\alpha}$ 
        iff $\SR_{\langle X \mid R \rangle}(\epsilon) \gg \epsilon^{\frac{1}{\alpha}}$. 
        
    \end{itemize}
    (with implied constants depending on $\langle X \mid R \rangle$). 
\end{ex}

Using Proposition \ref{HvsStabGrowthProp} we can now complete the proof 
of Theorem \ref{BeckMoshThmIntro} from the Introduction. 

\begin{thm}
    For all $d \geq 2$ there exists $D= D(d) > 0$ such that:
    \begin{center}
    $x^{\frac{d}{2}} \ll_d F_{\mathbb{Z}^d} ^{\pi} (x) \ll_d x^D$
    \end{center}
    for all $x \geq 1$. 
\end{thm}

\begin{proof}
    Fix $\langle X_d \mid R_d \rangle$ a finite presentation for $\Gamma_d \cong \mathbb{Z}^d$ 
    with associated epimorphism $\pi_d : \mathbb{F} \rightarrow \Gamma_d$. 
    By Theorem 1.16 of \cite{BecMos} there exists $D= D(d) > 0$ such that 
    $\SR_{\langle X_d \mid R_d \rangle}(\epsilon) \ll_d \epsilon ^{\frac{1}{D}}$. 
    By Example \ref{PolySREquiv} (i) we have $H_{\langle X \mid R \rangle}(x) \ll_d x^D$ 
    so that by Proposition \ref{HvsStabGrowthProp}, $F_{\Gamma_d} ^{\pi} (x) \ll_d x^D$. 

    For the converse inequality, by Theorem 1.17 of \cite{BecMos} we have 
    $\SR_{\langle X_d \mid R_d \rangle}(\epsilon) \gg_d \epsilon ^{\frac{1}{d}}$ so that 
    by Example \ref{PolySREquiv} (ii) we have $H_{\langle X_d \mid R_d \rangle}(x) \gg_d x^d$. 
    The group $\Gamma_d$ has quadratic Dehn function $D_{\langle X_d \mid R_d \rangle} (x) \approx x^2$ 
    (see for instance Corrolary 6.2 of \cite{Pit}), 
    so that the conditions of Remark \ref{BetterDehnBdRmrk} are satisfied, 
    hence $F_{\Gamma_d} ^{\pi} (x)^2 \gg_d x^d$. 
\end{proof}

\subsection{LEF growth and residual finiteness growth} \label{LEFRFPrelimSubsect}

For $\Gamma$ and $\Delta$ groups and $A \subseteq \Gamma$ a subset, 
a function $\phi : A \rightarrow \Delta$ is called a 
\emph{partial homomorphism} if, for all $g,h \in A$, 
if $gh\in A$ then $\phi(gh)=\phi(g)\phi(h)$. 
An injective partial homomorphism shall also be called 
a \emph{local embedding}. 
The group $\Gamma$ is \emph{locally embeddable into finite groups} 
(or \emph{LEF}) if, for every finite subset $A \subseteq \Gamma$, 
there exists a finite group $\Delta$ and a local embedding 
$A \rightarrow \Delta$. 
If $\Gamma$ is generated by a finite set $S$, 
then a prototypical finite subset of $\Gamma$ 
is the closed ball $B_S (l)$ of radius $l$ around the identity 
in the word-metric induced on $\Gamma$ by $S$. 
The \emph{LEF growth} function 
$\mathcal{L}_{\Gamma} ^S : \mathbb{N} \rightarrow \mathbb{N}$ 
of the finitely generated LEF group $\Gamma$ 
associated to the finite generating set $S$ is then 
defined such that for $l \in \mathbb{N}$, 
$\mathcal{L}_{\Gamma} ^S (l)$ is the minimal order of a finite 
group $\Delta$ affording a local embedding $B_S(l) \rightarrow \Delta$. 

The group $\Gamma$ is \emph{residually finite} if, 
for every finite subset $A \subseteq \Gamma$, 
there is a finite group $\Delta$ and a homomorphism $\Gamma\rightarrow\Delta$ 
whose restriction to $A$ is injective. 
In the case for which the residually finite group 
$\Gamma$ is generated by the finite set $S$, 
we define $\mathcal{R}_{\Gamma} ^S (l) \in \mathbb{N}$ 
to be the minimal order of such a finite group $\Delta$ 
in the case for which $A = B_S (l)$, 
leading to the definition of 
the \emph{full residual finiteness growth} function 
$\mathcal{R}_{\Gamma} ^S : \mathbb{N} \rightarrow \mathbb{N}$. 

It is immediate from these definitions 
that every residually finite group is LEF. 
If $\Gamma$ is generated by the finite set $S$ 
and $\Gamma$ is residually finite (respectively LEF) 
then $\mathcal{R}_{\Gamma} ^S$ (respectively $\mathcal{L}_{\Gamma} ^S$) 
is a nondecreasing function which is unbounded iff $\Gamma$ is infinite. 
Moreover for $\Gamma$ residually finite we have 
$\mathcal{L}_{\Gamma} ^S (l) \leq \mathcal{R}_{\Gamma} ^S (l)$ 
for all $l \in \mathbb{N}$. 
In some cases we have 
$\mathcal{L}_{\Gamma} ^S (l) = \mathcal{R}_{\Gamma} ^S (l)$ 
for all $l$ sufficiently large 
(for instance this occurs whenever $\Gamma$ is a finitely \emph{presented} 
group, see Corollary 2.21 of \cite{Brad}). 
Our interest shall be in the extreme opposite case: 
that for which $\mathcal{R}_{\Gamma} ^S$ 
grows much faster than $\mathcal{L}_{\Gamma} ^S$. 

\subsection{Making instability effective}

Suppose now that $\Gamma$ is LEF, and write $\pi(X)=S$. 
Suppose also that $F : (0,\infty) \rightarrow (0,\infty)$ satisfies 
$F_{\Gamma} ^{\pi} (1/\epsilon) \leq F(1/\epsilon)$. 
Then there exists a pair $(\delta,E)$ satisfying 
the conditions of Definition \ref{StableDefn} with 
$\lVert E \rVert \leq \lVert E \rVert / \delta \leq 2 F(1/\epsilon)$ 
(since $\delta \leq 1$). 
In particular, every element $r \in E$ satisfies 
$\lvert r \rvert \leq L(\epsilon ) := \max \big( 2 F(1/\epsilon), 1/2\epsilon \big)$. 
Now let $\psi : B_S \big( L(\epsilon ) \big) \rightarrow Q$ 
be a local embedding into a finite group $Q$. 
Composing with the regular representation of $Q$, 
we can replace $Q$ by $\Sym(\Omega)$ for some finite set $\Omega$ 
with $\lvert \Omega \rvert = \lvert Q \rvert$. 
and assume that for all $e \neq g \in B_S \big( L(\epsilon ) \big)$, 
$d_{\Omega} (\psi(g),\id_{\Omega}) = 1$. 
Then $x \mapsto (\psi \circ \pi)(x)$ defines a homomorphism 
$\rho : \mathbb{F} \rightarrow \Sym(\Omega)$ 
such that for every $w \in B_X \big( L(\epsilon ) \big)$, 
$\rho (w) = (\psi \circ \pi)(w)$. 
In particular, for all $r \in E$, $\rho(r)=\id_{\Omega}$, 
so that by stability, there exists a homomorphism 
$\phi : \Gamma \rightarrow \Sym(\Omega)$ such that 
$d_{\Omega} (\rho(x),(\phi \circ \pi)(x))=d_{\Omega} ((\psi \circ \pi)(x),(\phi \circ \pi)(x))<\epsilon$ for all $x \in X$. 

Consequently, for all $w \in B_X (1/2\epsilon)$, 
$d_{\Omega} ((\psi \circ \pi)(w),(\phi \circ \pi)(w))\leq 1/2$, 
and since $\pi \big( B_X(1/2\epsilon) \big) = B_S (1/2\epsilon)$, 
we have $\phi(g)\neq \id_{\Omega}$ for all $e \neq g \in B_S (1/2\epsilon)$, 
and we conclude that $\phi : \Gamma \rightarrow \Sym(\Omega)$ 
is a homomorphism whose restriction to $B_S (1/2\epsilon)$ is injective. 
It follows that: 
\begin{center}
$\lvert Q \rvert ! = \lvert \Sym(\Omega) \rvert \geq \mathcal{R}_{\Gamma} ^S (1/2\epsilon)$, 
\end{center}
where $\mathcal{R}_{\Gamma} ^S$ is the full 
residual finiteness growth function of $\Gamma$. 
Meanwhile, our only assumption on $Q$ was 
that it should admit an injective partial homomorphism 
from $ B_S \big( L(\epsilon ) \big)$. 
Taking $Q$ to be of minimal order among all such finite groups, we have: 
\begin{center}
$\lvert Q \rvert = \mathcal{L}_{\Gamma} ^S \big( L(\epsilon ) \big)$, 
\end{center}
where $\mathcal{L}_{\Gamma} ^S$ is the LEF growth function of $\Gamma$. 
Writing $\epsilon = 1/2l$, and since the above is valid for 
\emph{any} function $F$ satisfying $F_{\Gamma} ^{\pi} (1/\epsilon) \leq F(1/\epsilon)$, 
we can summarise our conclusion as follows. 

\begin{propn} \label{LEFStabProp}
Suppose $\Gamma$ is a finitely generated stable LEF group. 
Then ($\Gamma$ is residually finite and) for all $l \in \mathbb{N}$, 
\begin{equation} \label{LEFStabIneq}
\max \big\lbrace \mathcal{L}_{\Gamma} ^S \big( F_{\Gamma} ^{\pi} (2l) \big)!,\mathcal{L}_{\Gamma} ^S \big( l\big)! \big\rbrace
\geq \mathcal{R}_{\Gamma} ^S (l)
\end{equation}
\end{propn}

In other words, the better the stability growth function of $\Gamma$, 
the more tightly $\mathcal{R}_{\Gamma} ^S$ 
is controlled by $\mathcal{L}_{\Gamma} ^S$. 
For example, in polynomially stable groups, we get the following bound. 

\begin{coroll}
Let $\Gamma$ be a finitely generated stable LEF group, 
and suppose there exists $C\geq 1$ such that for all $\epsilon > 0$: 
\begin{center}
$F_{\Gamma} ^{\pi} (x) \leq Cx^C$. 
\end{center}
Then there exists $C' > 0$ such that for all $l \in \mathbb{N}$, 
\begin{center}
$\mathcal{L}_{\Gamma} ^S \big( C' l^C \big)! \geq \mathcal{R}_{\Gamma} ^S (l)$. 
\end{center}
\end{coroll}

Therefore to construct inefficiently stable groups, 
it is sufficient to find finitely generated LEF stable groups $\Gamma$ 
for which $\mathcal{R}_{\Gamma} ^S$ grows much faster than $\mathcal{L}_{\Gamma} ^S$. 

\begin{rmrk}
If $\pi : \mathbb{F} \twoheadrightarrow \Gamma$ is an isomorphism, 
then $F_{\Gamma} ^{\pi}$ is constant and $\mathcal{L}_{\Gamma} ^S = \mathcal{R}_{\Gamma} ^S$ 
grows exponentially (see \cite{Brad} Example 2.26, and using the fact that 
$\mathbb{F}$ is isomorphic to a subgroup of $\GL_2 (\mathbb{Z})$). 
If $\pi$ is \emph{not} an isomorphism, then by Proposition \ref{StabLinearProp} (ii) 
we have $F_{\Gamma} ^{\pi} (x) \geq x/4$ for all $x$ larger than some constant $C_1$. 
Thus we may simplify (\ref{LEFStabIneq}) to: 
$$\mathcal{L}_{\Gamma} ^S \big( F_{\Gamma} ^{\pi} (4l) \big)!\geq \mathcal{R}_{\Gamma} ^S (l)$$
for all $l$ sufficiently large. 
\end{rmrk}

\subsection{Quantitative stability and soficity}

The content of this Subsection is not required for the results of 
the subsequent Sections. 
As noted in Proposition \ref{LEFStabProp}, every stable LEF group is residually finite. 
More generally, it is well-known that residual finiteness 
holds for every stable \emph{sofic} group. 

\begin{defn} \label{SoficDefn}
The group $\Gamma$ is sofic if, for every finite 
subset $A \subseteq \Gamma$ and every $\epsilon > 0$, 
there exists a finite set $\Omega$ and a function $\phi : A \rightarrow \Sym(\Omega)$ 
satisfying for all $g,h \in A$: 
\begin{itemize}
    \item[(i)] If $gh \in A$ then $d_{\Omega} (\phi(gh),\phi(g)\phi(h)) < \epsilon$; 

    \item[(ii)] If $g \neq e$ then $d_{\Omega} (\phi(g),\id_{\Omega}) > 1-\epsilon$; 

    \item[(iii)] If $e \in A$ then $\phi(e) = \id_{\Omega}$. 
    
\end{itemize}
Such a map $\phi$ is called an \emph{$(A,\epsilon)$-almost representation} 
of $\Gamma$ on $\Omega$. 
\end{defn}

Arzhantseva and Cherix \cite{ArzChe} introduce a quantitative notion of soficity. 

\begin{defn}
    Let $\Gamma$ be a sofic group generated by a finite set $S$. 
    The \emph{sofic profile} $\mathcal{D}_{\Gamma} ^S : \mathbb{N} \rightarrow \mathbb{N}$ 
    is defined such that $\mathcal{D}_{\Gamma} ^S (l)$ is the minimum 
    value of $\lvert \Omega \rvert$ for which there exists 
    a $(B_S(l),1/l)$-almost representation
    $\phi : B_S(l) \rightarrow \Sym(\Omega)$. 
\end{defn}

\begin{ex}
    If $Q$ is a finite group; $\psi : B_S(l) \rightarrow Q$ is a local embedding, 
    and $\rho : Q \rightarrow \Sym(Q)$ is the regular representation of $Q$, 
    then $\rho \circ \psi$ is a $(B_S(l),\epsilon)$-representation for all $\epsilon > 0$. 
    Thus $\mathcal{D}_{\Gamma} ^S (l) \leq \mathcal{L}_{\Gamma} ^S (l)$ 
    (see also \cite{ArzChe} Subsection 4.3). 
\end{ex}

We may prove a generalisation of Proposition \ref{LEFStabProp} 
using sofic profile in place of LEF growth. 

\begin{propn}
Suppose that $\pi : \mathbb{F} \twoheadrightarrow \Gamma$ is not an isomorphism, 
and that $\Gamma$ is sofic. 
Then there exists $C>0$ such that for all $l \geq C$ we have: 
\begin{equation} \label{SoficProfStabRFIneq}
    \mathcal{D}_{\Gamma} ^{\pi(X)}\big(2 F_{\Gamma} ^X(10l)\big)! \geq \mathcal{R}_{\Gamma} ^{\pi(X)} (l). 
\end{equation}
\end{propn}

\begin{proof}
Write $F_{\Gamma} ^{\pi} = F$, and as before write $S = \pi (X)$. 
Set $\epsilon = 1/10l$ and let $(\delta,E)$ be a pair satisfying 
the conditions of Definition \ref{StableDefn} with 
$\lVert E \rVert / \delta \leq 2 F(10l)$. 
Let $\phi : B_S (2 F(10l)) \rightarrow \Sym(\Omega)$ be 
a $(B_S (2 F(10l)),1/2 F(10l))$-representation. 
Define a homomorphism $\psi : \mathbb{F} \rightarrow \Sym(\Omega)$ by 
$\psi(x) = (\phi \circ \pi)(x)$ for all $x \in X$. 
Then by condition (i) of Definition \ref{SoficDefn} and 
repeated use of the triangle inequality we have: 
\begin{equation} \label{ClosenessIneq}
d_{\Omega} \big( \psi(w),(\phi \circ \pi)(w) \big) < \frac{\lvert w \rvert}{2 F(10l)}
\end{equation}
for all $w \in B_X (2 F(10l))$. 
Specifically for $w \in E$ we have: 
$$\lvert w \rvert \leq \lVert E \rVert \leq \lVert E \rVert / \delta \leq 2 F(10l)$$
so that $w \in B_X (2 F(10l)) $ and: 
$$d_{\Omega} \big( \psi(w),\id_{\Omega} \big) < \delta$$
hence $\psi$ is a $(\delta,E)$-almost-solution. 
Let $\theta : \Gamma \rightarrow \Sym(\Omega)$ be a homomorphism 
such that $\theta \circ \pi$ and $\psi$ are $\epsilon$-close. 
We claim that the restriction of $\theta$ to $B_S (l)$ is injective, 
so that $\mathcal{R}_{\Gamma} ^S (l) \leq \lvert \Omega \rvert !$. 
Taking $\Omega$ to be of minimal size then yields (\ref{SoficProfStabRFIneq}) 
and concludes the proof. 

Let $e \neq g \in B_S (l)$ and let $w \in B_X (l)$ with $\pi(w) = g$. 
Then: 
$$d_{\Omega} \big( (\theta \circ \pi)(w) , \psi(w) \big) < \epsilon \lvert w \rvert \leq \frac{1}{10} $$
so that by (\ref{ClosenessIneq}) we have: 
\begin{align*}
d_{\Omega} \big( \theta(g) , \id_{\Omega} \big) 
& \geq d_{\Omega} \big( \phi(g) , \id_{\Omega} \big) 
- d_{\Omega} \big( \psi(w),(\phi \circ \pi)(w) \big) 
- d_{\Omega} \big( \theta \circ \pi)(w) , \psi(w) \big) \\
& > \frac{9}{10} - \frac{1}{10l} - \frac{l}{2 F(10l)} \\
& \geq \frac{6}{10}
\end{align*}
where the last inequality holds for all $l$ sufficiently large by 
Proposition \ref{StabLinearProp} (ii). 
Thus $\theta(g) \neq \id_{\Omega}$, as desired. 
\end{proof}

\section{Generalized B.H. Neumann groups} \label{NeumannSection}

\subsection{Definitions and basic properties} \label{NeumannBasicSubsect}

Let $d , r  : \mathbb{N} \rightarrow \mathbb{N}$ 
be functions such that $d$ is nondecreasing and 
takes only odd values at least $5$, 
and $d(n) \geq 2r(n)+1$ for all $n \in \mathbb{N}$. 
Let: $$\alpha_n = \big(1 \; 2 \; \cdots \; d(n)\big) , 
\beta_n = \big(1 \; (1+r(n)) \; (1+2r(n))\big) \in \Alt (d(n)),$$
let $\alpha = (\alpha_n) , \beta = (\beta_n) \in \prod_n \Alt(d(n))$ and let 
$S = \lbrace \alpha,\beta \rbrace$. 
Let: $$G = G(d,r) = \langle S \rangle \leq \prod_n \Alt(d(n))$$ 
be the \emph{generalized B.H. Neumann group} associated 
to the functions $d$ and $r$. 
The classical \emph{B.H. Neumann groups} 
correspond to the special case for which $r \equiv 1$ (see \cite{Neum}). 
We denote by $\pi_m$ the natural projection 
$\prod_n \Alt(d(n)) \twoheadrightarrow \Alt(d(m))$, 
so that $\pi_m (\alpha) = \alpha_m$ and $\pi_m (\beta) = \beta_m$. 

In general, it may happen that $\pi_m (G)$ is a proper subgroup 
of $\Alt(d(m))$. For instance if $d(m)=3r(m)$ then 
$\alpha_m$ and $\beta_m$ both preserve the nontrivial block-system 
of residue classes modulo $3$, hence $\pi_m (G)$ 
is imprimitive. 
We may avoid such cases by assuming that $d(m)$ is a prime number. 

\begin{lem} \label{PiSurjLem}
Suppose $d(m) \geq 5$ is prime. 
Then $\langle \alpha_m,\beta_m \rangle = \Alt(d(m))$. 
\end{lem}

\begin{proof}
Note first that, since $d(m)$ is prime and $\alpha_m$ is a $d(m)$-cycle, 
$\langle \alpha_m,\beta_m \rangle$ is a primitive subgroup of $ \Alt(d(m))$. 
The result then follows from Jordan's Theorem: 
if $d \geq 5$ and $H \leq \Sym(d)$ is a primitive subgroup 
containing a $p$-cycle for some prime $p \leq d-2$ then $\Alt(d) \leq H$. 
Since $\alpha_m$ and $\beta_m$ are even permutations and 
$\beta_m$ is a $3$-cycle, we are done. 
\end{proof}

Let $L_m = 
\langle \alpha^j \beta \alpha^{-j} : \lvert j \rvert\leq m \rangle 
\leq G$ and set: 
\begin{equation*}
L_{\infty} = \bigcup_m L_m
\end{equation*}
so that $L_{\infty}$ is the normal closure in $G$ of 
the element $\beta$, 
and $G / L_{\infty} \cong \mathbb{Z}$ is generated by 
$\alpha L_{\infty}$. 
Recall that for $\theta \in \Sym(n)$, 
$\supp(\theta)$ is the number of points $1 \leq i \leq n$ 
satisfying $\theta(i) \neq i$. 

\begin{lem} \label{LSmallSuppLem}
Let $l \in L_m$. Then for all $n \in \mathbb{N}$, 
$\supp(\pi_n (l)) \leq 6m+3$. 
\end{lem}

\begin{proof}
The permutation group $\pi_n (L_m)$ is generated by 
the 3-cycles $\alpha_n ^j \beta_n \alpha_n ^{-j}$, for $\lvert j \rvert \leq m$
Thus: 
\begin{equation*}
\supp(\pi_n (l)) \leq \sum_{\lvert j \rvert \leq m} \supp (\alpha_n ^j \beta_n \alpha_n ^{-j}) \leq 3(2m+1)
\end{equation*}
as desired. 
\end{proof}

\begin{lem}
If $\langle \alpha_n , \beta_n \rangle = \Alt(d(n))$, 
then $\pi_n(L_m) = \Alt(d(n))$ for all $m \geq d(n) \log d(n) / \log(2)$. 
\end{lem}

\begin{proof}
We claim first that if $m \in \mathbb{N}$ with $\pi_n(L_m) \neq \Alt(d(n))$ 
then $\pi_n(L_m) \lneq \pi_n(L_{m+1})$. 
Suppose to the contrary that $\pi_n(L_m) = \pi_n(L_{m+1}) \neq \Alt(d(n))$. 
Then for all $\lvert j \rvert\leq m$, 
\begin{center}
$\alpha_n ^{\pm 1} (\alpha_n ^j \beta_n \alpha_n ^{-j}) \alpha_n ^{\mp 1} 
\in \pi_n(L_{m+1}) = \pi_n(L_m)$ and 
$\beta_n ^{\pm 1} (\alpha_n ^j \beta_n \alpha_n ^{-j}) \beta_n ^{\mp 1} 
\in \pi_n(L_m)$. 
\end{center}
Thus, since $\langle \alpha_n , \beta_n \rangle = \Alt(d(n))$ 
we have that $\pi_n(L_m)$ is a nontrivial normal subgroup 
of $\Alt(d(n))$, contradicting the simplicity of the latter. 

Hence the $\pi_n(L_m)$ form a strictly increasing sequence of 
nontrivial subgroups of $\Alt(d(n))$, 
so $\lvert \pi_n(L_m) \rvert \geq 2^m$ by induction. 
The conclusion follows since $\lvert \Alt(d(n))\rvert \leq d(n)^{d(n)}$. 
\end{proof}

\subsection{The lamplighter quotient} \label{LampSubsect}

Let $G = G(d,r)$ and $S$ be as in Subsection \ref{NeumannBasicSubsect}. 
Throughout this Subsection we make the following standing assumptions: 
\begin{itemize}
\item[(i)] $d$ takes only prime values; 
\item[(ii)] $r$ is strictly increasing; 
\item[(iii)] For all $n \in \mathbb{N}$, $3r(n) \leq d(n)$. 
\end{itemize}
Let $W = C_3 \wr \mathbb{Z} = \bigoplus_{\mathbb{Z}} C_3 \rtimes \mathbb{Z}$ be the regular restricted wreath product 
of $C_3$ and $\mathbb{Z}$, 
so that $\mathbb{Z}$ acts on 
$B(W)=\bigoplus_{\mathbb{Z}} C_3$ 
by shifting co-ordinates. 
Let $a$ be a generator for $\mathbb{Z}$ 
and for $n \in \mathbb{Z}$ let $b_n \in B(W)$ 
be a generator for the copy of $C_3$ supported at $n$, 
such that $a^n b_m a^{-n} = b_{m+n}$. 
Note that $\lbrace a,b_0 \rbrace$ 
is a generating set for $W$ and 
$\lbrace b_n :n\in\mathbb{Z}\rbrace$ 
is a generating set for $B(W)$. 

The next observation is Proposition 2.13 of \cite{BradRF}, 
specialized to the case $r_1 = r_2 = r$. 

\begin{propn} \label{AltWreathLocalProp}
Let $w(x, y) \in F (x, y)$ be a freely reduced word of length at
most $l$ in the variables $x$ and $y$. Let $n \in \mathbb{N}$ and suppose: 
\begin{equation}
r(n) , d(n) - 2r(n) \geq 2l+1. 
\end{equation}
Then $w(\alpha_n,\beta_n) = e$ in $\Alt(d(n))$ 
iff $w(a,b_0)=e$ in $W$. 
\end{propn}

A first, crucial application of Proposition \ref{AltWreathLocalProp} 
is a decomposition of many generalized B.H. Neumann groups 
as a short exact sequence with quotient $W$ and 
kernel a direct sum of finite alternating groups. 

\begin{propn} \label{TailHomProp}
There exists a surjective homomorphism 
$\tau : G \rightarrow W$ satisfying 
$\tau(\alpha) = a$ and $\tau (\beta) = b_0$. 
The kernel of $\tau$ is: 
\begin{equation*}
\ker (\tau) = \bigoplus_n \Alt (d(n)) \leq \prod_n \Alt (d(n))
\end{equation*}
(in particular $\bigoplus_n \Alt (d(n))$ is contained in $G$)
and $\ker (\tau) \leq L_{\infty}$. 
\end{propn}

\begin{proof}
By our standing assumptions, 
$r(m)$ and $d(m) - 2r(m)$ both tend to infinity with $m$. 
The first statement, 
and the fact that $$\ker (\tau) = G \cap \bigoplus_n \Alt (d(n))$$ 
are then immediate from our Proposition \ref{AltWreathLocalProp} 
and Lemma 4.6 of \cite{KasPak}. 
For the remainder of the statement, 
we prove by induction on $m$ that $$\bigoplus_{n=1} ^m \Alt (d(n)) 
\leq L_{\infty}.$$
Since $r$ is strictly increasing, we have that 
$x_m = [\alpha^{r(m)}\beta\alpha^{-r(m)} , \beta] \in L_{\infty}$ satisfies: 
\begin{center}
$\pi_n (x_m) = e$ for $n > m$ but $\pi_m (x_m) \neq e$. 
\end{center}
Let $X_m$ be the normal closure in $L_{\infty}$ of $x_m$
Then $X_m \leq \bigoplus_{n=1} ^m \Alt (d(n))$ and, 
since $\pi_m(L_{\infty}) = \Alt(d(m))$ 
(by Lemma \ref{PiSurjLem}) 
we have $\pi_m(X_m) = \Alt(d(m))$ (by simplicity of $\Alt(d(m))$). 
By inductive hypothesis we have $\bigoplus_{n=1} ^{m-1} \Alt (d(n)) \leq L_{\infty}$, 
which, together with $X_m$, generates $\bigoplus_{n=1} ^m \Alt (d(n))$, 
yielding the desired result. 
\end{proof}

Another observation, 
closely related in spirit to Proposition \ref{AltWreathLocalProp} 
is the following, which we shall also require. 

\begin{propn} \label{3CycProp}
Let $m,n \in \mathbb{N}$ and suppose $r(n) , d(n) - 2r(n) \geq 2m+1$. 
For any $l \in L_m$, $\tau(l)=e$ iff $\pi_n(l)=e$. 
\end{propn}

\begin{proof}
Let $F(Y)$ be the free group on the set 
$Y = \lbrace y_{i} : \lvert i \rvert \leq m \rbrace$. 
Then there is some word $w(y_{-m},\ldots,y_m) \in F(Y)$ 
such that $l=w(\alpha^{-m}\beta\alpha^m ,\ldots ,\alpha ^m \beta\alpha^{-m})$ 
so that $\tau(l)=w(b_{-m},\ldots,b_m)$ and 
$\pi_n(l)=w(\alpha_n ^{-m} \beta_n \alpha_n ^m , \ldots , \alpha_n ^m \beta_n \alpha_n ^{-m})$. 
The elements $b_{-m},\ldots,b_m \in B(W)$ form a 
minimal generating set for an elementary abelian $3$-group, 
so it suffices to show that 
$\alpha_n ^{-m} \beta_n \alpha_n ^m , \ldots , \alpha_n ^m \beta_n \alpha_n ^{-m} \in \Alt(d(n))$ also do so. 
This is the case, because they are $3$-cycles, 
which by our hypothesis 
are all disjointly supported. 
\end{proof}

\subsection{Amenable structure}

We retain the standing assumptions from the previous Subection. 
Recall that for $G$ a group generated by a finite symmetric set $T$, 
a \emph{F\o lner sequence} for $G$ is a sequence $(F_n)$ 
of nonempty finite subsets of $G$ satisfying: 
$$ \frac{\lvert F_n t \setminus F_n \rvert}{\lvert F_n \rvert} \rightarrow 0 \text{ as }n \rightarrow \infty $$
for all $t \in T$. Such a group $G$ is \emph{amenable} 
if it has a F\o lner sequence. 
Every finitely generated (locally finite)-by-soluble 
group is known to be amenable, 
so those generalized B.H. Neumann groups which satisfy 
the conclusion of Proposition \ref{TailHomProp} are seen to be amenable. 
It shall however be important to our arguments to have 
an explicit F\o lner sequence. 
Let $L_m = 
\langle \alpha^j \beta \alpha^{-j} : \lvert j \rvert\leq m \rangle 
\leq G$ be as described in Subsection \ref{NeumannBasicSubsect}. 

\begin{lem}
For all $m \in \mathbb{N}$, 
$L_m$ is a finite subgroup of $G$. 
\end{lem}

\begin{proof}
Let $\tau$ be as in Proposition \ref{TailHomProp}. 
Then $\tau(L_m) = \langle a^j b a^{-j} : \lvert j \rvert\leq m \rangle \leq B(W)$ is a finitely generated subgroup of the abelian torsion group 
$B(W)$, hence is finite. 
Thus $L_m \cap \ker (\tau)$ has finite index in $L_m$, 
hence is finitely generated. 
But by Proposition \ref{TailHomProp}, $\ker(\tau)$ is locally finite, 
so $L_m \cap \ker (\tau)$ is finite. 
\end{proof}

Let  $S = \lbrace \alpha,\beta \rbrace$ 
be as in Subsection \ref{NeumannBasicSubsect}. 
For $n \geq 2$ let: 
$$G_n = \bigoplus_{k=1} ^{n-1} \Alt (d(k))$$
(which, by Proposition \ref{TailHomProp}, is a finite normal subgroup of $G$). 
Fix an increasing sequence $(m_n)$ of positive integers, 
and set:
$$F_n = \lbrace g l \alpha^j : g \in G_n , l \in L_{m_n} , \lvert j \rvert \leq m_n \rbrace$$
so that $G = \bigcup_n F_n$. 

\begin{propn} \label{NeumFolnerProp}
The sequence $(F_n)$ is a F\o lner sequence for $G$. 
Consequently $G$ is an amenable group. 
\end{propn}

\begin{proof}
We shall check the definition of F\o lner sequence 
given above applies, with $T$ the symmetric closure of $S$. 
Note that $G_n L_{m_n}$ is a finite subgroup of $L_{\infty}$. 
Let $g \in G_n$, $l \in L_{m_n}$ and $\lvert j \rvert \leq m_n$. 
Then $(gl \alpha^j) \beta^{\pm 1} = gl (\alpha^j \beta \alpha^{-j})^{\pm 1} \alpha^j \in F_n$, and since this holds for all such $g$, $l$ and $j$ 
we have $F_n \beta = F_n$. 
Meanwhile $(gl \alpha^j)\alpha = gl \alpha^{j+1}$, 
which lies in $F_n$ unless $j=m_n$, hence: 
$$ \frac{\lvert F_n \alpha \setminus F_n \rvert}{\lvert F_n \rvert} 
= \frac{\lvert G_n L_{m_n} \rvert}{(2m_n +1)\lvert G_n L_{m_n} \rvert} = \frac{1}{2m_n +1}$$
(since the cosets $G_n L_{m_n} \alpha^i$ are distinct as $i$ ranges over $\mathbb{Z}$). 
Similarly: 
$$\frac{\lvert F_n\alpha^{-1} \setminus F_n}{\lvert F_n \rvert} \rvert
\leq \frac{1}{2m_n +1}$$ and we have the desired conclusion. 
\end{proof}

\begin{lem} \label{FolnerConjLem}
Let $g \in L_{\infty}$ and for $n \in \mathbb{N}$ set: 
$$E_n (g) = \lbrace f \in F_n : fgf^{-1} \in G_n L_{m_n} \rbrace.$$
Then $\lvert E_n (g) \rvert / \lvert F_n \rvert \rightarrow 1$ 
as $n \rightarrow \infty$. 
\end{lem}

\begin{proof}
Let $m \in \mathbb{N}$ be such that $g \in L_m$, 
let $n$ be sufficiently large that $m_n > m$, 
let $h \in G_n$ and $l \in L_{m_n}$. 

Note that $\alpha^{\pm 1} L_m \alpha^{\mp 1} \leq L_{m+1}$. 
Thus $(l \alpha^j) g (l \alpha^j)^{-1} \in L_{m_n}$ for all 
$l \in L_{m_n}$ and $\lvert j \rvert \leq m_n - m$. For such $j$, 
$$(hl \alpha^j) g (hl \alpha^j)^{-1} = [h^{-1},((l \alpha^j) g (l \alpha^j)^{-1})^{-1}]\cdot (l \alpha^j) g (l \alpha^j)^{-1} $$
which lies in $G_n L_{m_n}$ since $G_n$ is normal in $G$. 
Since all $\alpha^i$ lie in distinct cosets of $G_n L_{m_n}$, 
we have: $$\lvert E \rvert \geq (2(m_n-m)+1) \lvert G_n L_{m_n} \rvert \text{ and } \lvert F_{m_n} \rvert = (2m_n + 1)\lvert G_n L_{m_n} \rvert,$$
whence the required bound. 
\end{proof}

Let $W_m = C_3 \wr (\mathbb{Z}/(2m + 1)\mathbb{Z})$ 
and set $B(W_m) = \bigoplus_{\mathbb{Z}/(2m + 1)\mathbb{Z}} C_3 \vartriangleleft W_m$. 
Let $a_{(m)}$ be a generator for $\mathbb{Z}/(2m + 1)\mathbb{Z}$, 
and for $-m \leq i \leq m$ let $b_{(m),i} \in B(W_m)$ 
be a generator for the copy of $C_3$ supported at 
$i+(2m + 1)\mathbb{Z}$, so that 
$a_{(m)} b_{(m),i} a_{(m)} ^{-1} = b_{(m),j}$ with 
$j \equiv i + m \mod (2m+1)$. 
Then there is a unique surjective homomorphism 
$\rho_m : W \rightarrow W_m$ 
satisfying $\rho_m(a) = a_{(m)}$ and $\rho_m(b_0) = b_{(m),0}$. 
It is easy to see that, 
for $(F_n)$ the F\o lner sequence for $G$ described above, 
the restriction of $\rho_{m_n}$ to $\tau(F_n)$ is a bijection onto $W_{m_n}$. 
Recall that $G_n =\Alt(d(1))\times\ldots\times\Alt(d(n-1)) \leq \ker(\tau)$. 

\begin{lem} \label{FutureHomLem}
Let $P_n = G_n \times W_{m_n}$ and let: 
$$\phi_n =(\pi_1 \times\ldots\times\pi_{n-1} \times (\rho_{m_n}\circ \tau)) : 
G \rightarrow P_n.$$
Then the restriction $\phi_n |_{F_n} : F_n \rightarrow P_n$ 
of $\phi_n$ to $F_n$ is a bijection, 
and $\phi_n (G_n L_{m_n}) = G_n \times B(W_{m_n}) \leq P_n$.  
\end{lem}

\begin{proof}
We shall prove that $\phi_n$ sends $G_n L_{m_n} \leq G$ 
bijectively onto $G_n \times B(W_{m_n})$. 
Since $F_n$ is the disjoint union of the cosets 
$G_n L_{m_n} \alpha^j$ (for $\lvert j \rvert \leq m_n$), 
$G_n \times B(W_{m_n})$ has index $(2m_n + 1)$ in $P_n$ 
and the $\phi_n (\alpha^j) \in G_n \times \lbrace a_{(m_n),j} \rbrace$ lie in distinct cosets of $G_n \times B(W_{m_n})$ in $P_n$, 
it follows from this that $\phi_n$ maps $F_n$ bijectively onto $P_n$. 

For surjectivity, 
note that the elements 
$\alpha^j \beta \alpha^{-j} \in L_{m_n}$ (for $\lvert j \rvert \leq m_n$) 
are mapped by $(\rho_{m_n}\circ \tau)$ to $b_{(m_n),j}$, 
which generate $B(W_{m_n})$. 
Thus given $(g,b) \in G_n \times B(W_{m_n})$, 
we may take $l \in L_{m_n}$ with $\phi_n (l) = (h,b)$ 
for some $h \in G_n$. 
On the other hand, 
$(\pi_1 \times\ldots\times\pi_{n-1})$ sends $G_n$ identically to itself, 
and since $G_n \leq \ker(\tau)$, 
it follows that $\phi_n (gh^{-1}l) = (g,b)$. 

For injectivity, 
let $g \in G_n$, $l \in L_{m_n}$ and suppose $\phi_n (gl) = e$. 
Since $\tau(g) = e$, 
we have $(\rho_{m_n}\circ \tau)(gl)=(\rho_{m_n}\circ \tau)(l) = e$. 
But $\rho_{m_n}$ is injective on $\tau(L_{m_n})$, so $\tau(l)=e$. 
By Proposition \ref{3CycProp}, for all $k \neq n$
$\pi_{k} (gl) = \pi_k (l) = e$. 
But as $\phi_n (gl) = e$, we also have $\pi_k(gl)=e$ for $k \leq n-1$, 
hence $gl=e$. 
\end{proof}

\subsection{Approximation properties}

A second application of Proposition \ref{AltWreathLocalProp} is 
an upper bound on the LEF growth of many 
generalized B.H. Neumann groups. 

\begin{propn} \label{LEFUBProp}
There exists $C > 0$ such that, if $d$ and $r$ satisfy: 
$$r(n),d(n)-2r(n) \geq n$$
for all $n \in \mathbb{N}$, then: 
\begin{equation*}
\mathcal{L}_G ^S (l) \leq \exp (Cl^2 \log(l))
\end{equation*}
for all $l \in \mathbb{N}$. 
\end{propn}

The conclusion of Proposition \ref{LEFUBProp} 
is complementary to Theorem 4.9 of \cite{Brad}, 
where an upper bound of the order of $\exp(\exp(l))$ 
was given on the LEF growth of a certain 
family of generalized B.H. Neumann groups 
(those constructed by Bou-Rabee and Seward in \cite{BoRaSewa}). 
The common ingredient of both bounds is the next Lemma, 
which is proved as Example 4.3 of \cite{Brad}. 
Our bound here is stronger than in the result of \cite{Brad}, 
owing to the availability of Proposition \ref{AltWreathLocalProp}. 

\begin{lem} \label{AltLocalLem}
Let $d \geq 3$ and let $t_1 , t_2 \in \Sym(d)$ be 
an $n$-cycle and a $3$-cycle, respectively. 
For all $l,d' \in \mathbb{N}$, if $d' \geq 9 l + 3$ then 
there is a local embedding $B_T(l) \rightarrow \Sym(d')$, 
where $T = \lbrace t_1,t_2 \rbrace$. 
\end{lem}

\begin{proof}[Proof of Proposition \ref{LEFUBProp}]
Fix $l \in \mathbb{N}$ and set 
$Q_l = \Sym(15l)^{4l+1}$, 
so that $$\lvert Q_l \rvert = ((15l)!)^{4l+1}l \leq \exp (75l^2 \log(15l)) 
\leq \exp (Cl^2 \log(l))$$ for $C>0$ a sufficiently large constant. 
It therefore suffices to show that there is a local embedding 
$B_S(l) \rightarrow Q_l$. 
First, let: $$\pi = (\pi_1 \times \cdots \times \pi_{4l+1}) : G \rightarrow \prod_{n=1} ^{4l+1} \Alt(d(n)); $$
we claim that the restriction of $\pi$ to $B_S (l)$ is injective. 
For if $e \neq g \in \ker (\pi) \cap B_S (l)$, 
then there exists $n > 4l+1$ such that $\pi_n (g) \neq e$. 
Letting $w \in F(x,y)$ be a nontrivial reduced word of length at most $l$ 
such that $g=w(\alpha,\beta)$, 
we have from Proposition \ref{AltWreathLocalProp} 
that $w(\alpha_{4l+1},\beta_{4l+1}) = e$ in $\Alt(d(4l+1))$ 
(since $g \in \ker(\pi) \leq \ker(\pi_{4l+1})$) 
so that $w(a,b_0) = e$ in $W$ 
and hence $w(\alpha_n,\beta_n) = e$ in $\Alt(d(n))$, which is a contradiction, 
and the claim is proved. 

We note also that: $$\pi(B_S(l)) 
\subseteq B_{\pi_1(S)} (l) \times \cdots \times B_{\pi_{4l+1}(S)} (l),$$
thus it suffices to construct, for each $1 \leq n \leq 4l+1$, 
a local embedding $\rho_n : B_{\pi_n(S)} (l) \rightarrow \Sym(15l)$, 
so that: $$(\rho_1 \times \cdots \times \rho_{4l+1}) \circ \pi|_{B_S(l)} 
: B_S(l) \rightarrow Q_l$$
is the desired local embedding. 
For $n \leq 4l$, we can define $\rho_n$ using Lemma \ref{AltLocalLem} 
(if $d(n) < 9l + 3$, we simply take $\rho_n$ to be inclusion). 
For the case $n=4l+1$ let: $$\overline{\alpha} = \big(1 \; 2 \; \cdots \; 12l+3 \big) , \overline{\beta} = \big(1 \; (4l+2) \; (8l+3)\big) \in \Alt(12l+3).$$
Then by Proposition \ref{AltWreathLocalProp} again, 
for any freely reduced word $v(x,y) \in F(x,y)$ of length at most $2l$, 
$w(\alpha_{4l+1},\beta_{4l+1}) = e$ in $\Alt(d(4l+1))$ 
iff $w(a,b_0) = e$ in $W$ 
iff $w(\overline{\alpha},\overline{\beta}) = e$ in $\Alt(12l+3)$. 
It follows that there is a local embedding 
$\rho_{4l+1} : B_{\pi_{4l+1}(S)} (l) \rightarrow \Alt(12l+3)$ 
given by $\rho_{4l+1} (w(\alpha_{4l+1},\beta_{4l+1})) = w(\overline{\alpha},\overline{\beta})$. 
\end{proof}

Next, we have our main stability result, 
the proof of which is deferred until the next Section. 

\begin{thm} \label{MainStabilityThm}
Suppose $d$ and $r$ satisfy: 
(i) $d$ takes only prime values; 
(ii) $r$ is strictly increasing, and 
(iii) for all $n \in \mathbb{N}$, $3r(n) \leq d(n)$. 
Then $G$ is stable. 
\end{thm}

The generalized B.H. Neumann groups, 
being subdirect products of finite groups, 
are residually finite. 
They include, however, groups of arbitrarily fast 
(full) residual finiteness growth. 
This was the content of the main result of \cite{BoRaSewa}; 
here we follow the treatment given in \cite{BradRF}. 

\begin{thm} \label{RFGrowthThm}
There exists an absolute constant $C_0 > 0$ such that 
for any increasing function $F : \mathbb{N} \rightarrow \mathbb{N}$ 
there exist functions $d , r : \mathbb{N} \rightarrow \mathbb{N}$ 
satisfying the hypotheses of 
Proposition \ref{LEFUBProp} and Theorem \ref{MainStabilityThm}, 
and such that for all $n \geq C_0$, 
\begin{equation*}
\mathcal{R}_G ^S (n) \geq F(n). 
\end{equation*}
\end{thm}

\begin{proof}
For $n \in \mathbb{N}$ set $q(n)=9n^2$, $p(n)=4(q(n)+17n)+4$ 
(a strictly increasing function) 
and let $C>0$ be a sufficiently large constant 
(to be chosen in the course of the proof). 
We define $d:\mathbb{N} \rightarrow \mathbb{N}$ recursively. 
Let $d(1)$ be a prime with $d(1) \geq C$. 
Given $d(1),\ldots,d(n-1)$, 
let $d(n)$ be a prime satisfying: 
$$d(n) \geq \max \big( d(n-1)+1 , Cn^2 , F(p(n+1)) \big)$$
so that (assuming $C$ is sufficiently large), 
$d$ and $q$ satisfy the hypotheses of Proposition 2.9 from \cite{BradRF}. 
We deduce that there exists a function 
$r : \mathbb{N}\rightarrow \mathbb{N}$ satisfying: 
\begin{itemize}
\item[(a)] For all $n \in \mathbb{N}$, 
$q(n) < r(n) < q(n) +  17 n$ and $r(n) < d(n)/3$; 

\item[(b)] For all $l,m\in \mathbb{N}$, 
if $l \neq m$ then 
$r(l) \nequiv \pm r(m),\pm 2 r(m) \mod d(m)$. 
\end{itemize}
In particular, by choice of $q$, we have $q(n+1)>r(n)$ for all $n$, 
hence $r$ is strictly increasing. 
By Proposition 3.5 of \cite{BradRF}, 
for such $d$ and $r$, $G = G(d,r)$ satisfies: 
$$\mathcal{F}_G ^S(4r(n)+4) \geq d(n)! / 2$$
for all $n$, where $\mathcal{F}_G ^S(l)$ denotes 
the minimal value $M$ such that every nontrivial element 
of $B_S(l)$ lies outside the kernel of some homomorphism 
from $G$ to some finite group of order at most $M$. 
Thus $\mathcal{R}_G ^S(l) \geq \mathcal{F}_G ^S(l)$. 
For $m \geq p(1) = 108$, let $n \in \mathbb{N}$ be such that 
$p(n+1) \geq m \geq p(n)$, then: 
\begin{align*}
\mathcal{R}_G ^S(m) \geq \mathcal{R}_G ^S(p(n)) \geq \mathcal{R}_G ^S(4r(n)+4) 
\geq d(n)! / 2 \geq d(n) \geq F(p(n+1)) \geq F(m)
\end{align*}
since $p(n) \geq 4r(n)+4$. 
\end{proof}

The goal of the next section shall be to prove Theorem \ref{MainStabilityThm}. 
For now we assume Theorem \ref{MainStabilityThm} 
and complete the proof of our Main Theorem from the Introduction. 

\begin{thm}
There exists an absolute constant $C_1 > 0$ such that 
for any nondecreasing function $F : [1,\infty) \rightarrow (0,\infty)$ 
there exists a two-generated stable group $\Gamma$ and a 
surjective homomorphism $\pi : \mathbb{F}_2 \twoheadrightarrow \Gamma$ 
such that for all $x \geq C_1$, 
$F_{\Gamma} ^{\pi} (x) \geq F(x)$. 
\end{thm}

\begin{proof}
Let $C>0$ be as in Proposition \ref{LEFUBProp}. 
Enlarging $F$ if required, we may assume that 
for all $l \in \mathbb{N}$ we have 
\begin{equation}
\exp \big( CF(2l-1)^2 \log (F(2l-1))\big)! > \exp \big( Cl^2 \log (l)\big)!
\end{equation}
By Theorems \ref{MainStabilityThm} and \ref{RFGrowthThm} 
there exist an absolute constant $C_0 > 0$ and functions $d$ and $r$ 
satisfying $r(n),d(n)-2r(n) \geq n$ for all $n \in \mathbb{N}$, 
such that $G=G(d,r)$ is stable and for all $l \geq C_0$, 
\begin{equation}
\mathcal{R}_{\Gamma} ^S (l) \geq \exp \big( CF(2l-1)^2 \log (F(2l-1))\big)!
\end{equation}
so that, applying Propositions \ref{LEFStabProp} and \ref{LEFUBProp}, 
we obtain $F_{\Gamma} ^{\pi} (2l) \geq F(2l-1)$, 
and hence $F_{\Gamma} ^{\pi} (l) \geq F(l)$ for all $l \geq C_0$, 
where $\Gamma = G$ and $\pi$ sends the free generators of 
$\mathbb{F}_2$ to $\alpha$ and $\beta$, respectively. 
\end{proof}

\section{Invariant random subgroups} \label{IRSSect}

For $\Gamma$ a countable group, 
let $\Sub(\Gamma) \subseteq \lbrace 0,1 \rbrace^{\Gamma}$ 
be the space of subgroups of $\Gamma$, 
equipped with the Tychonoff topology. 
There is a continuous action of $\Gamma$ on $\Sub(\Gamma)$ 
by conjugation, and an \emph{invariant random subgroup} 
(or \emph{IRS} for short) 
of $\Gamma$ is a $\Gamma$-invariant Borel probability 
measure on $\Sub(\Gamma)$. 
By $\IRS(\Gamma)$ we denote the space 
of all IRS of $\Gamma$. 
This is a compact metrizable space 
under the weak$^{\ast}$ topology, 
and forms a Choquet simplex within $\Prob (\Gamma)$. 
An IRS of $\Gamma$ is \emph{ergodic} if it is an extreme 
point of this simplex: that is, if it cannot 
be written as a proper convex combination 
of other elements of $\IRS(\Gamma)$. 
Thus there are no proper closed convex subspace of 
$\IRS(\Gamma)$ containing all ergodic IRS of $\Gamma$. 

An IRS of $\Gamma$ is called \emph{finite-index} 
if it is an atomic probability measure 
supported on finite-index subgroups of $\Gamma$, 
and is called \emph{cosofic} in $\Gamma$ if it is 
the weak$^{\ast}$-limit of finite-index IRS of $\Gamma$. 

\begin{thm}[Theorem 1.3 of \cite{BecLubTho}] \label{IRSStabCritThm}
Let $\Gamma$ be a finitely generated amenable group. 
If every ergodic IRS of $\Gamma$ is cosofic in $\Gamma$ 
then $\Gamma$ is stable. 
\end{thm}

For $\Delta \leq \Gamma$ a finite-index 
subgroup, a \emph{finite-to-one transversal} 
to $\Delta$ in $\Gamma$ is a finite set $T \subseteq \Gamma$ 
which is a disjoint union of transversals to $\Delta$ 
in $\Gamma$. 

\begin{lem} \label{EasyTransLem}
Let $H$ and $K$ be finite-index subgroups of the group $\Gamma$, 
with $K \leq H$. Suppose $T$ is a finite-to-one 
transversal to $K$ in $\Gamma$. 
Then $T$ is a finite-to-one 
transversal to $H$ in $\Gamma$. 
\end{lem}

We shall use the following tool for identifying cosofic 
IRS of amenable groups, 
which is a combination of Proposition 3.3 and Theorem 3.4 
from \cite{LevLub}. 

\begin{thm} \label{WeissThm}
Let $\Gamma$ be a finitely generated amenable group, 
and let $(F_n)$ be a F\o lner sequence in $\Gamma$. 
Let $\mu \in \IRS (\Gamma)$. 
Suppose that for $\mu$-almost every 
subgroup $H$ of $\Gamma$, 
there exists a sequence $(K_n)$ of finite-index subgroups 
of $\Gamma$ such that: 
\begin{itemize}
\item[(i)] For all $n$, $F_n$ is a finite-to-one transversal 
to $N_{\Gamma}(K_n)$; 

\item[(ii)] For all $g \in \Gamma$, 
\begin{equation}
p_n (g) = \frac{\lvert \lbrace f \in F_n : fgf^{-1} \in H \smalltriangleup K_n \rbrace \rvert}{\lvert F_n \rvert} 
\rightarrow 0 \text{ as }n \rightarrow \infty. 
\end{equation}
\end{itemize}
Then $\mu$ is cosofic in $\Gamma$. 
\end{thm}

Now let $G = G(d,r)$ be as in Subsection \ref{NeumannBasicSubsect}, 
with $d,r : \mathbb{N} \rightarrow \mathbb{N}$ 
satisfying the conditions of Theorem \ref{MainStabilityThm}. 

\begin{thm} \label{BasicIRSThm}
Let $\mu \in \IRS (G)$ be an ergodic IRS and suppose that 
$\supp (\mu) \subseteq \Sub( L_{\infty})$. Then $\mu$ is cosofic. 
\end{thm}

\begin{proof}
Suppose each coset of $K$ in $\Gamma$ contains $n$ elements of $T$. 
Then each coset of $H$ in $\Gamma$ contains $n \lvert H:K \rvert$ 
elements of $T$, since every coset of $H$ in $\Gamma$ 
is a union of $\lvert H:K \rvert$ cosets of $K$ in $\Gamma$. 
\end{proof}

\begin{proof}[Proof of Theorem \ref{BasicIRSThm}]
We shall use the criterion from Theorem \ref{WeissThm}. 
Let $(m_n)$ be an increasing sequence of positive integers. 
Let $(F_n)$ be as in Section \ref{NeumannSection}, 
so that, by Proposition \ref{NeumFolnerProp}, 
$(F_n)$ is a F\o lner sequence for $G$. 
Let $H \leq L_{\infty}$ and let $\phi_n : G \rightarrow G_n \times W_{m_n}$ 
be as in Lemma \ref{FutureHomLem}, 
define $N_n = \ker (\phi_{m_n})$ 
(a finite-index normal subgroup of $G$) and set 
$K_n = (H \cap (G_n L_{m_n})) N_n \leq G$, 
so that $K_n$ is a finite-index subgroup of $G$. 
We verify the conditions of Theorem \ref{WeissThm}. 

For hypothesis (i), 
recall from Lemma \ref{FutureHomLem} that $F_n$ 
is a transversal to $N_n$ in $G$. 
We have $N_n \leq K_n \leq N_G (K_n)$, 
so b Lemma \ref{EasyTransLem}, 
$F_n$ is a finite-to-one transversal to $N_G (K_n)$ in $G$. 

For hypothesis (ii) we divide into two cases. 
First suppose that $g \notin L_{\infty}$. 
Let $0 \neq k \in \mathbb{Z}$ be such that 
$g L_{\infty} = \alpha^k L_{\infty}$. 
Then $(fgf^{-1}) L_{\infty} = \alpha^k L_{\infty}$ also 
(since $G/ L_{\infty} \cong \mathbb{Z}$ is abelian). 
Thus: $$(\rho_{m_n} \circ \tau)(fgf^{-1}) \in a_{(m_n)} ^k B(W_{m_n}) \neq B(W_{m_n})$$ 
for all $m_n > k$. 
On the other hand, if $fgf^{-1} \in H \smalltriangleup K_n$ then, 
since $fgf^{-1} \notin L_{\infty} \geq H$, we have $fgf^{-1} \in K_n$. 
But $\rho_{m_n} (K_n) \leq B(W_{m_n})$ for all $n$. 
Hence for all $m_n > k$ we have $fgf^{-1} \notin H \smalltriangleup K_n$ 
for all $f \in F_n$, that is $p_n(g) = 0$ for all $n$ sufficiently large. 

Alternatively, suppose $g \in L_{\infty}$. 
Let $E_n (g) = \lbrace f \in F_n : fgf^{-1} \in G_n L_{m_n} \rbrace$. 
Then we have: 
\begin{equation} \label{GoodFsEqn}
\frac{\lvert E_n \rvert}{\lvert F_n \rvert} \rightarrow 1 \text{ as } 
n \rightarrow \infty
\end{equation}
by Lemma \ref{FolnerConjLem}. 
For $f \in E_n (g)$, we have $fgf^{-1} \notin H \setminus K_n$ 
since: $$(G_n L_{m_n}) \cap H \leq (G_n L_{m_n}) \cap K_n.$$
On the other hand, if $fgf^{-1} \in K_n$ then, 
writing $fgf^{-1} = hd$ for some $h \in (G_n L_{m_n}) \cap H$ and $d \in N_n$, 
we must have $d \in G_n L_{m_n}$ also (since $fgf^{-1} \in G_n L_{m_n}$). 
But from Lemma \ref{FutureHomLem}, $(G_n L_{m_n}) \cap N_n = \lbrace e \rbrace$ 
so $fgf^{-1} = h \in (G_n L_{m_n}) \cap H$. 
Thus $fgf^{-1} \notin K_n \setminus H$ also. Hence: 
\begin{equation*}
p_n(g) \leq 1 - \frac{\lvert E_n \rvert}{\lvert F_{m_n} \rvert} \rightarrow 0 
\text{ as } n\rightarrow \infty \text{ (by (\ref{GoodFsEqn})). }
\end{equation*}
Since this holds for every subgroup $H$ of $L_{\infty}$, 
the desired result now follows from Theorem \ref{WeissThm}. 
\end{proof}

We now turn our attention to ergodic IRS of $G$ 
which are not supported on $L_{\infty}$. 
It transpires that these IRS are not just cosofic, 
they are actually finite-index. 
Recall that a subgroup $H$ of a group $\Gamma$ 
is \emph{almost-normal} 
if $N_{\Gamma}(H)$ has finite index in $\Gamma$. 
Note that if $\mu \in \IRS (\Gamma)$ and 
$H \leq \Gamma$ is an \emph{atom} of $\mu$, 
meaning that $\mu (\lbrace H \rbrace) > 0$, 
then since: 
\begin{center}
$\mu (\lbrace H \rbrace) = g_{\ast}\mu (\lbrace H \rbrace) = \mu (\lbrace H^{g^{-1}} \rbrace)$
\end{center}
we must have that $H$ is almost-normal in $\Gamma$. 

Let $B(W) = \bigoplus_{\mathbb{Z}} C_3 \vartriangleleft W = C_3 \wr \mathbb{Z}$ be as in the preceding section. 

\begin{thm}[Lemma 2.2 of \cite{BowGrigKrav}] \label{BowGrigKravThm}
There are only countably many subgroups 
of $W$ not contained in $B(W)$. 
\end{thm}

\begin{propn} \label{WreathANProp}
If $H \leq W$ is almost-normal and $H \nsubseteq B(W)$, 
then $\lvert W : H \rvert < \infty$. 
\end{propn}

\begin{proof}
Suppose $N_W (H)$ has index $A$ in $W$. 
Then $\lvert B(W) : N_{B(W)} (H) \rvert \leq A$, 
so we may take a set $v_1 , \ldots , v_A \in B(W)$ 
of coset representatives to $N_{B(W)} (H)$ in $B(W)$. 
Thus there exists $C > 0$ such that 
$v_i \in \langle b_j : \lvert j \rvert \leq C \rangle$ for $1 \leq i \leq A$. 

There exists a positive integer $k$ and $v,w \in B(W)$ 
such that $x = v a^k = a^k w \in H$. 
Let $F = \langle b_j : \lvert j \rvert \leq C+k \rangle$, 
a finite subgroup of $B(W)$. 
For any $j \in \mathbb{Z}$, there exists $1 \leq i(j) \leq A$ 
such that $b_j v_{i(j)} ^{-1} \in N_{B(W)} (H)$. 
We therefore have: 
\begin{center}
$x^{-1} (b_j v_{i(j)} ^{-1}) x (b_j v_{i(j)} ^{-1})^{-1} 
= (a^{-k} b_j a^k b_j ^{-1})(a^{-k} v_{i(j)} ^{-1} a^k v_{i(j)}) 
\in H$
\end{center}
and since $a^{-k} v_{i(j)} ^{-1} a^k v_{i(j)}\in F$ 
we have: 
$$b_{j+k} b_j ^{-1} = a^{-k} b_i a^k b_i ^{-1} \in (H \cap B(W))F 
\text{ for all $j \in \mathbb{Z}$.}$$
It is easy to see that these elements, together with $b_j \in F$ 
($1 \leq j \leq k$), generate $B(W)$. 
Since $(H \cap B(W)),F \leq B(W)$, an abelian group, we have $(H \cap B(W))F \leq B(W)$. 
Thus, in fact, $(H \cap B(W))F = B(W)$, 
and we conclude that the finite set 
$\lbrace f a^i : f \in F , 0 \leq i \leq k-1 \rbrace$ 
contains a set of coset representatives to $H$ in $W$. 
\end{proof}


Our next observation, and its proof, 
are closely analogous to \cite{LevLub} Proposition 2.5. 

\begin{propn} \label{WreathLiftProp}
Let $H \leq G_d$. 
Suppose that $\lvert W : \tau(H) \rvert < \infty$. 
Then $\lvert G_d : H \rvert < \infty$. 
\end{propn}

To prove Proposition \ref{WreathLiftProp} 
we shall require the 
following fact from the theory of permutation groups, 
originally due to Babai (see Theorems 5.3A and 5.4A of \cite{DixMor}). 

\begin{thm} \label{BabaiThm}
Let $\Omega$ be a finite set, 
and let $G \leq \Sym(\Omega)$ be a primitive subgroup. 
Suppose there exists $e \neq g \in G$ with 
$\supp(g) < \sqrt{\lvert \Omega \rvert}/2$. 
Then $\Alt(\Omega) \leq G$. 
\end{thm}

\begin{lem} \label{WreathLiftLem}
Let $H \leq G_d$. 
Suppose that $\lvert W : \tau(H) \rvert < \infty$. Then: 
\begin{itemize}
\item[(i)] There exists $n_1 \in \mathbb{N}$ such that for all $n > n_1$, 
$\pi_n (H) = \Alt(d(n))$; 

\item[(ii)] If $S \subseteq \mathbb{N}$ is a set of integers 
which is not cofinite, then 
$$Q_S = H / H \cap \big( \bigoplus_{n \in S}\Alt(d(n)) \big) $$
is not soluble. 

\end{itemize}
\end{lem}

\begin{proof}
Since $\tau(H)$ has finite index in $W$, 
there exist $e \neq v \in B(W) \cap \tau(H)$ and 
$0 \neq k \in \mathbb{Z}$ such that $a^k 
= \tau(\alpha^k) \in \tau(H)$. 
Since $\tau(L_{\infty}) = B(W)$, 
there exists $m > 0$ and $\tilde{v} \in L_m$ 
such that $\tau(\tilde{v}) = v$. 
As $v \neq e$ we have from Proposition \ref{AltWreathLocalProp} 
that for all $n$ sufficiently large, $\pi_n (\tilde{v}) \neq e$. 
Since $\ker(\tau) = \bigoplus_n \Alt(d(n))$, 
there exists $n_0 \in \mathbb{N}$ such that 
for all $n \geq n_0$: 
\begin{itemize}
\item[(i)] $d(n) > \lvert k \rvert$, 

\item[(ii)] $\alpha_n ^k \in \pi_n (H)$ and $\pi_n (\tilde{v}) \neq e$. 

\end{itemize}
By the assumption that $d(n)$ is a prime greater than $k$, 
$\alpha_n ^k \in \pi_n (H)$ is a $d(n)$-cycle, 
so that $\pi_n (H)$ is transitive, 
hence (since, once again, $d(n)$ is prime) is primitive. 
From Lemma \ref{LSmallSuppLem} we have $\supp(\pi_n (\tilde{v})) \leq 6m+3$ 
for all $n$.  
Let $n_1 \geq n_0$ with $d(n_1) > (12m+6)^2$, 
so that Theorem \ref{BabaiThm} applies 
to $g = \pi_n(\tilde{v}) \in \pi_n (H) \leq \Alt(d(n))$ for all $n \geq n_1$
and we have (i). 

For (ii), we note that by (i) there exists $n_1 \in \mathbb{N} \setminus S$ 
such that $\pi_{n_1} (H) = \Alt(d(n_1))$. 
Since $\bigoplus_{n \in S}\Alt(d(n)) \leq \ker (\pi_{n_1})$, 
we have that $Q_S$ has $\Alt(d(n_1))$ as a homomorphic image, 
hence is not soluble. 
\end{proof}

\begin{proof}[Proof of Proposition \ref{WreathLiftProp}]
Given Lemma \ref{WreathLiftLem}, we can proceed much as in the proof of 
\cite{LevLub} Proposition 2.5. 
Letting $n_1$ be as in Lemma \ref{WreathLiftLem} (i), 
and taking $H_0 \vartriangleleft H$ to be the kernel of: 
$$ (\pi_1 \times \cdots \times \pi_{n_1})|_H : H \rightarrow \Alt(d(1)) \times \cdots \times \Alt(d(n_1)),$$ 
we have $\pi_n (H_0) = \Alt(d(n))$ for all $n > n_1$, 
as a consequence of the Jordan-H\"{o}lder Theorem. 
Moreover $\tau(H_0)$ still has finite index in $W$, 
so it suffices to check that $H_0$ is of finite index in $G$. 

Next, we check that $H_0 \cap \ker(\tau) \vartriangleleft \ker(\tau)$. 
For let $h \in H_0 \cap \ker(\tau)$ and let $g \in \ker(\tau)$. 
We claim that $ghg^{-1} \in H_0 \cap \ker(\tau)$. 
By definition of $H_0$, $\pi_n (h) = e$ for $n \leq n_1$. 
On the other hand, since $h \in \ker(\tau)$, there exists $n_2 \geq n_1$ 
such that $\pi_n (h) = e$ for all $n > n_2$. 
Since $\pi_n (H_0) = \Alt(d(n))$ for $n_1 +1 \leq n \leq n_2$, 
then by a standard consequence of the Jordan-H\"{o}lder Theorem, 
we have that: 
$$ (\pi_{n_1 + 1} \times \cdots \times \pi_{n_2})|_{H_0} : H_0 \rightarrow \Alt(d(n_1 + 1)) \times \cdots \times \Alt(d(n_2))$$ 
is surjective. Consequently, there exists $\tilde{g} \in H_0$ 
such that $\pi_n (\tilde{g}) = \pi_n (g)$ for all 
$n_1 +1 \leq n \leq n_2$. 
Thus $g h g^{-1} = \tilde{g} h \tilde{g}^{-1} \in H_0$, as desired. 
From Proposition \ref{TailHomProp}, 
$$\ker(\tau) = \bigoplus_{n \in \mathbb{N}} \Alt(d(n))$$
and by simplicity of the $\Alt(d(n))$, 
any normal subgroup of $\ker(\tau)$ is a direct sum 
of some subfamily of the $\Alt(d(n))$. 
Thus there exists $S \subseteq \mathbb{N}$ such that 
$$ H_0 \cap \ker(\tau) = \bigoplus_{n \in S} \Alt(d(n)). $$
As $H_0 / H_0 \cap \ker(\tau)$ is isomorphic to 
a subgroup of $W$, it is soluble, 
hence by Lemma \ref{WreathLiftLem} (ii) (applied to $H_0$) 
we conclude that $S$ is cofinite in $\mathbb{N}$. 
Thus $\lvert \ker(\tau) : H_0 \cap \ker(\tau) \rvert , \lvert W : \tau(H_0) \rvert < \infty$, so $\lvert G : H_0 \rvert < \infty$ also.   
\end{proof}

Finally we put everything together and prove Theorem \ref{MainStabilityThm}. 

\begin{proof}
By Theorem \ref{IRSStabCritThm}, 
it suffices to prove that, 
if $\mu \in \IRS (G)$ is 
an ergodic invariant random subgroup of $G$, 
then $\mu$ is cosofic in $G$. 
If $\supp(\mu) \subseteq \Sub( L_{\infty})$, 
then Theorem \ref{BasicIRSThm} applies, 
so suppose not. 

We claim that $\mu$ is a finite-index IRS (hence cosofic). 
To this end, let $T \subseteq \Sub(G)$ 
(respectively $\overline{T} \subseteq \Sub(W)$) 
be the set of infinite-index subgroups of $G$ (respectively of $W$). 
As $G$ and $W$ are finitely-generated, $T$ and $\overline{T}$ 
are cocountable, hence Borel. 
Consider the pushforward IRS $\tau_{\ast} \mu \in \IRS(W)$. 
By Proposition \ref{WreathLiftProp}, 
$\tau$ sends each member of $T$ into $\overline{T}$, 
so $\mu(T) \leq (\tau_{\ast} \mu)(\overline{T})$, 
hence it suffices to check $(\tau_{\ast} \mu)(\overline{T}) = 0$. 

By normality of $L_{\infty}$ in $G$, 
we have that $\Sub( L_{\infty})$ is a $G$-invariant Borel subset 
of $\Sub(G)$, so by ergodicity of $\mu$, 
we have: $$\supp(\mu) \subseteq U = \Sub(G) \setminus \Sub( L_{\infty}).$$
Let $\overline{V} = \Sub(B(W)) \subseteq \Sub(W)$ 
and let $\overline{U} = \Sub(W) \setminus \Sub(B(W))$. 
Then by Proposition \ref{AltWreathLocalProp}, 
$\tau$ sends $U$ and $\Sub( L_{\infty})$ 
into $\overline{U}$ and $\overline{V}$, respectively 
(since a subgroup $H$ of $G$ lies outside $L_{\infty}$ iff it contains 
$\alpha^k$ for some $k > 0$ iff 
$\tau(H)$ contains $a^k$ for some $k > 0$ iff $\tau(H)$ lies outside $B(W)$). 
Thus $\tau_{\ast} \mu$ is supported on $\overline{U}$. 
By Theorem \ref{BowGrigKravThm}, 
$\tau_{\ast} \mu$ is atomic, hence supported on almost-normal subgroups 
in $\overline{U}$. 
By Proposition \ref{WreathANProp}, 
these all have finite index, 
hence $(\tau_{\ast}\mu)(\overline{T})=0$, 
as desired. 
\end{proof}

\section{Further questions} 

In the light of the connections to property-testing problems, 
it is especially natural to attempt to quantify stability for 
finite sets of relations. 
Alas, the generalised B.H. Neumann groups are typically not finitely presented. 
For those groups in the family for which the epimorphism to $C_3 \wr \mathbb{Z}$ 
from Subsection \ref{LampSubsect} is defined, 
this follows immediately from the main result of \cite{GenTes}, 
which implies that no finitely presented amenable group admits an 
epimorphism to a wreath product $A \wr B$, with $B$ infinite and $A$ nontrivial. 
More generally, any finitely generated group for which the full residual 
finiteness growth is not eventually equal to the LEF growth is not 
finitely presented (Proposition 2.18 of \cite{Brad}). 
As such, the following remains open, 
where we refer to the group $\Gamma$ as \emph{polynomially stable} 
if $F_{\Gamma} ^{\pi} (x)$ is bounded above by a polynomial function of $x$ 
(see also Problem II.16 of \cite{BecLubMos}). 

\begin{qu}
    Does there exist a finitely \emph{presented} group $\Gamma$ which is stable 
    but not polynomially stable? 
\end{qu}

Similarly, one wishes to have more examples of polynomially stable groups. 
In this vein, we amplify and generalise a suggestion from Section 6 of \cite{BecMos}. 

\begin{qu} \label{NilpQu}
    Is every finitely generated nilpotent group polynomially stable? 
    What about every polycyclic group? 
\end{qu}

A positive answer to (either part of) Question \ref{NilpQu} 
would also yield a positive answer to Problem 2 of \cite{BecLubMosI}, 
which asks for a first example of a nonabelian polynomially stable group. 
Note that all finitely generated nilpotent groups are finitely presented 
and satisfy a polynomial Dehn function \cite{GerHolRil},  
so that, using Proposition \ref{HvsStabGrowthProp} as we did in the proof of 
Theorem \ref{BeckMoshThmIntro}, the property of being polynomially stable 
in our sense is equivalent for such groups $\Gamma=\langle X \mid R \rangle$ 
to having a stability rate 
$\SR_{\langle X \mid R \rangle}(\epsilon)$ bounded above by 
$\epsilon^{\frac{1}{\alpha}}$ for some $\alpha \geq 1$. 

Finally, recall that for a group $\Gamma$ to be stable means that the ``sample-and-substitute'' 
algorithm described in the Introduction is capable of recognising 
tuples of permutations which are close to a solution 
for the family of defining relations for $\Gamma$. 
A priori, the property that proximity to a solution 
is recognised by ``sample-and-substitute'' (stability) is stronger than 
the property of recognition by \emph{some} algorithm. 
Groups satisfying the latter, weaker property are known as \emph{testable} groups \cite{BecLubMos}. 
Not every testable group is stable: for instance every finitely generated amenable 
group is testable (Theorem 1.14 of \cite{BecLubMosI}). 
It is natural to ask of a testable group $\Gamma$ for bounds on the complexity 
of algorithms witnessing the testability of the group. 
Remarkably, there exists a ``universal'' algorithm, 
called in \cite{BecLubMos} the \emph{Local Statistics Matcher} algorithm, 
which witnesses testability for any finitely generated testable group. 
Therefore, one could define a notion of \emph{testability growth function} 
for a finitely generated testable group $\Gamma$ 
in one of two ways: either as a measure of the 
complexity of the \emph{most efficient} algorithm for recognising which 
tuples of permutations are close to a solution for the relations of $\Gamma$, 
or as a measure of the complexity of Local Statistics Matcher when applied to $\Gamma$. 
The following questions may be interpreted in either sense. 

\begin{qu}
    Does there exist a finitely generated group which is not stable but 
    which is polynomially testable? Does there exist a finitely presented example? 
\end{qu}

\begin{qu} \label{BadTestableQu}
    Does there exist a finitely generated group which is testable but not 
    polynomially testable? Does there exist a stable such group? 
    Does there exist a finitely presented such group? 
    Do there exist finitely generated groups 
    with arbitrarily large testability growth functions? 
    Do there exist such groups which are stable? 
\end{qu}

The tools that we have employed in this paper to give obstructions to effective stability 
(specifically large full residual finiteness growth) 
do not seem so relevant to Question \ref{BadTestableQu}, 
because sofic testable groups need not be residually finite 
(for instance, there exist many finitely generated amenable groups which are not 
residually finite). 
One therefore seeks other group-theoretic properties, which are obstructions to testability, 
and which may be quantified. 

\subsection*{Acknowledgments} I am grateful to Oren Becker; Alon Dogon; Francesco Fournier-Facio and Tianyi Zheng 
for enlightening conversations which helped shape the development of this paper.


\begin{thebibliography}{99}

\bibitem{ArzChe} G. Arzhantseva, P.-A. Cherix, 
{\it Quantifying metric approximations of discrete groups. }
Ann. Fac. Sci. Toulouse Math (2024), 831-877.

\bibitem{ArzPau} G. Arzhantseva, L. Paunescu, 
{\it Almost commuting permutations are near commuting permutations. }
Journal of Functional Analysis, 269, no. 3 (2015): 745-757.

\bibitem{BecLubTho} O. Becker, A. Lubotzky, A. Thom, 
{\it Stability and invariant random subgroups. }
Duke Mathematical Journal, 168, no. 3 (2019), 2207-2234. 

\bibitem{BecLubMos} O. Becker, A. Lubotzky, J. Moshieff, 
{\it Testability of relations between permutations. }
2021 IEEE 62nd annual symposium on foundations of computer science (FOCS). IEEE (2022).

\bibitem{BecLubMosI} O. Becker, A. Lubotzky, J. Moshieff, 
{\it Testability in group theory. }
Israel Journal of Mathematics, 256 (2023), 61–90. 

\bibitem{BecMos} O. Becker, J. Moshieff, 
{\it Abelian groups are polynomially stable. }
International Mathematics Research Notices, 20 (2021), 15574-15632. 

\bibitem{BowGrigKrav} L. Bowen, R. Grigorchuk, R. Kravchenko, 
{\it Invariant random subgroups of lamplighter groups. }
Israel Journal of Mathematics 207.2 (2015), 763-782.

\bibitem{BradRF} H. Bradford, 
{\it The inverse problem for residual finiteness growth. }
arXiv:2402.03556 [math.GR]

\bibitem{Brad} H. Bradford, 
{\it Quantifying local embeddings into finite groups. }
Journal of Algebra 608 (2022), 214-238.

\bibitem{BoRaSewa} K. Bou-Rabee, B. Seward, 
{\it Arbitrarily large residual finiteness growth. }
J. Reine Angwe. Math. 2016 (710), 199-204.

\bibitem{DixMor} J.D. Dixon, B. Mortimer, 
{\it Permutation groups. }
Graduate Texts in Mathematics 163, 
Springer Science \& Business Media (2012). 

\bibitem{DogLevVig} A. Dogon, A. Levit, I. Vigdorovich, 
{\it Characters of diagonal products and Hilbert-Schmidt stability. }
arXiv preprint arXiv:2407.11608 (2024).

\bibitem{GenTes} A. Genevois, R. Tessera, 
{\it A note on morphisms to wreath products. }
athematical Proceedings of the Cambridge Philosophical Society (2025). 

\bibitem{GerHolRil} S. Gersten, D.F. Holt, T.R. Riley, 
{\it Isoperimetric inequalities for nilpotent groups.}
Geometric and Functional Analysis 13.4 (2003), 795-814.

\bibitem{KasPak} M. Kassabov, I. Pak, 
{\it Groups of oscillating intermediate growth. }
Annals of Mathematics (2013), 1113-1145.

\bibitem{LevLub} A. Levit, A. Lubotzky, 
{\it Uncountably many permutation stable groups. }
Israel Journal of Mathematics 251.2 (2022), 657-678. 

\bibitem{Neum} B.H. Neumann, 
{\it Some remarks on infinite groups. }
J. London Math. Soc. 12 (1937) 120-127. 

\bibitem{Pit} C. Pittet, 
{\it Isoperimetric inequalities in nilpotent groups. }
J. London Math. Soc. 55.3 (1997): 588-600.

\end{thebibliography}
\end{document}